\documentclass{elsarticle}
\usepackage[letterpaper,margin=1in]{geometry}

\usepackage{bm}
\usepackage{tikz-cd}
\usepackage{amsfonts}
\usepackage{graphicx}
\usepackage{algorithmic}
\usepackage{amsmath,amssymb}
\usepackage{mathrsfs}
\usepackage{stmaryrd}
\usepackage{amsthm}
\usepackage{lineno}
\usepackage{hyperref}
\usepackage{algorithm}
\usepackage{cleveref}
\usepackage{booktabs}
\usepackage{multirow}
\usepackage{placeins}

\numberwithin{algorithm}{section}

\newtheorem{theorem}{Theorem}[section]

\newtheorem{corollary}[theorem]{Corollary}

\newtheorem{remark}[theorem]{Remark}

\crefname{hypothesis}{Hypothesis}{Hypotheses}

\numberwithin{equation}{section}
\numberwithin{figure}{section}
\numberwithin{table}{section}

\journal{Computer Methods in Applied Mechanics and Engineering}
\begin{document}

\begin{frontmatter}

\title{Smoother-type a posteriori error estimates for finite element methods}

\author[1]{Yuwen Li\corref{corresponding}}
\ead{liyuwen@zju.edu.cn}

\author[1]{Han Shui}
\ead{shuihan@zju.edu.cn}
\cortext[corresponding]{Corresponding author.}

\affiliation[1]{organization={School of Mathematical Sciences, Zhejiang University},
            addressline={866 Yuhangtang Road},
            city={Hangzhou},
            postcode={310058},
            state={Zhejiang},
            country={People's Republic of China}}
            
\tnotetext[2]{\textbf{Funding:}~ This work was supported by the National Key R\&D Program of China under grant 2024YFA1012600 and the National Natural Science Foundation of China under grant 12471346.}
            
\begin{abstract}
This work develops user-friendly a posteriori error estimates of finite element methods, based on smoothers of linear iterative solvers. The proposed method employs simple smoothers, such as Jacobi or Gauss--Seidel iteration, on an auxiliary finer mesh to process the finite element residual for a posteriori error control. The implementation has linear complexity and requires only a coarse-to-fine prolongation operator. For symmetric problems, we prove the reliability and efficiency of smoother-type error estimators under a saturation assumption. Numerical experiments for various PDEs demonstrate that the proposed smoother-type error estimators outperform residual-type estimators in accuracy and exhibit robustness with respect to physical parameters and polynomial degrees.
\end{abstract}


\end{frontmatter}

\section{Introduction}
Adaptive mesh refinement is a common routine in competitive numerical methods for resolving singularities and sharp gradients as well as saving computational cost in large-scale numerical simulations. In Adaptive Finite Element Methods (AFEMs), a posteriori error estimation is a key tool for equi-distributing errors and constructing nearly optimal locally refined meshes, see \cite{Schwab1998,BabuskaStrouboulis2001,Verfurth2013,BonitoCanutoNochettoVeeser2024} for
basic ingredients of AFEMs and references therein. Classical a posteriori error estimates of AFEMs include residual, superconvergent recovery, dual weighted residual, equilibrated, hierarchical basis, functional a posteriori error estimators, see, e.g., \cite{ZienkiewiczZhu1992b,Bank1996,AinsworthOden2000,BeckerRannacher2001,BankXu2003a,CarstensenBartels2002,ZhangNaga2005,Ainsworth2007,Repin2008,ErnVohralik2015,ErnVohralik2020,Li2025arxiv,Li2018SINUM,BankLi2019,Li2021JSCb} for an incomplete list of references in this field.

A posteriori error estimation of Finite Element Methods (FEMs) is also closely related to iterative linear solvers because both methods aim at developing computable approximations to the action of solution operators on some residuals. For instance, the idea of hierarchical basis solvers motivated the development of hierarchical basis error estimators in \cite{BankSmith1993,Bank1996}. Multigrid solvers combined with superconvergence are used for constructing a posteriori error estimates in \cite{BankXu2003b,BankXuZheng2007}. In \cite{MulitaGianiHeltai2021}, residual error estimators are applied to inaccurate intermediate solutions produced by several smoothing passes to save computational cost in AFEMs. It has been shown in 
\cite{LiZikatanov2021CAMWA,LiZikatanov2025mcom} that additive operator preconditioning on infinite-dimensional spaces yields implicit and parameter-robust a posteriori error estimates for $H({\rm grad})$,
$H({\rm curl})$ and $H({\rm div})$ problems. Conversely, a posteriori error estimation has been applied to design coarse spaces in multigrid methods and to control iterative algebraic errors in domain decomposition methods, see \cite{XuZikatanov2018,HuWuZikatanov2021,BastidasVohralik2025}.

We shall follow the research line in \cite{LiZikatanov2021CAMWA,LiZikatanov2025mcom} to derive a posteriori error estimates based on smoothers of iterative linear solvers. Smoothers are simple basic iterations such as Jacobi or Gauss--Seidel method on a fine mesh and are efficient for damping the high frequency component of the FEM error. We shall show that appropriate smoothers directly yield high-quality a posteriori error estimates of FEMs for various PDEs such as the Poisson, curl-curl, Biot, convection-diffusion and Helmholtz equations. In numerical experiments, smoother-based error estimators exhibit sharper effectivity ratios than residual-type a posteriori error estimates. In contrast to the additive preconditioning-type implicit error estimators derived in \cite{LiZikatanov2021CAMWA,LiZikatanov2025mcom}, which involve incomputable local problems, the proposed smoother-type a posteriori error estimators are fully computable. Moreover, this work introduces a new family of error estimators inspired by contractive iteration methods, which are particularly well-suited for handling non-symmetric problems. {For Lagrange finite element discretizations, the smoother-type error estimators are explicit, as they avoid  solving patchwise local problems.}

The implementation of the smoother-type a posteriori error estimates on a mesh $\mathcal{T}_h$ is similar to a two-grid solver. In particular, we shall make use of a simple smoother $S$, e.g., a Jacobi, block Jacobi or Gauss--Seidel iteration on an auxiliary finer mesh $\mathcal{T}_{h/2}$. Then a smoother-type a posteriori error estimate is of the form $\|Sr\|$ or $\sqrt{\langle r,Sr\rangle}$, where $r$ is the finite element residual transferred to $\mathcal{T}_{h/2}$, $\|\bullet\|$ is the norm of interest and $\langle\bullet,\bullet\rangle$ is the pairing between the primal and dual spaces. The reliability and efficiency of our a posteriori error estimates are confirmed under a saturation assumption.  Using the same assumption, adaptive finite and boundary element methods driven by the difference $\|u_h-u_{h/2}\|$ of finite element solution on nested meshes have been analyzed in \cite{FerrazPraetorius2008,FerrazOrtnerPraetorious2010}, while computing the finer mesh solution $u_{h/2}$ is much more expensive than the smoother action $Sr$.

Smoother-type a posteriori error analysis in this paper offers a possibility that directly translates available iterative solvers of complex PDE systems into a posteriori error estimates inheriting many desirable properties of established solvers. For instance, complex multi-physics problems are modeled by saddle-point systems depending on a set of physical parameters (e.g., Lamé constants and permeability). In the numerical study, a common first step is a discrete inf-sup stability analysis and a uniformly efficient block diagonal preconditioner. The proposed framework avoids redoing a posteriori error analysis for multi-physics systems and is able to rewrite their solvers as parameter-robust smoother-type a posteriori error estimates, see Section \ref{subsect:Biot} for an application in Biot's consolidation model.

In addition, numerical experiments demonstrate that the effectiveness of smoother-type a posteriori error estimates is robust with respect to the polynomial degree (referred to as $p$-robust) of finite elements. The $p$-robustness is an important property in $hp$-AFEMs and real-world simulations by very high order FEMs (cf.~\cite{Schwab1998,Demkowicz2007,Demkowicz2008}). In the literature, well-known $p$-robust a posteriori error estimates include equilibrated and superconvergent error estimators based on solving local problems on each vertex patch, see \cite{BraessPillweinSchoberl2009,BartelsCarstensen2002}. Compared to classical $p$-robust estimators, pointwise smoother-type error estimators are easier to implement and do not necessitate solving auxiliary problems on local patches.

The rest of the paper is organized as follows. In
Section~\ref{sect:framework} we provide preliminaries
and set up the abstract framework. In Section~\ref{sect:discreteV} we employ smoothers of  iterative solvers to derive a posteriori error estimates for positive-definite problems.  Section~\ref{sect:saddle}
is devoted to the smoother-type a posteriori error estimates of symmetric saddle-point systems. Section~\ref{sect:nonsymmetric} illustrates the applications of smoother-type estimators to convection-diffusion and Helmholtz equations.  Section \ref{sect:conclusion} is the concluding remarks.

\section{Abstract Framework}\label{sect:framework}
In this section, we present two smoother-type a posteriori error estimates for an abstract positive-definite model posed in a Hilbert space that admits a stable subspace decomposition.

\subsection{Abstract Error Estimators}
Let $V_h\subset V$ be a pair of nested Hilbert spaces and $a: V\times V\rightarrow \mathbb{R}$ be a bounded and inf-sup stable bilinear form. Given $f\in V^*$ in the dual space $V^*$, we consider the following model problem and its Galerkin discretization: Find $u\in V$ and $u_h\in V_h$ such that
\begin{subequations}\label{eq:model}
    \begin{align}
    a(u,v)&=f(v),\quad v\in V,\label{subeq:model}\\
    a(u_h,v_h)&=f(v_h),\quad v_h\in V_h.\label{subeq:FEM}
\end{align}
\end{subequations}
Throughout this paper, $V_h$ is a proper finite element space based on a partition $\mathcal{T}_h$ of the underlying domain $\Omega$ for \eqref{subeq:model}.
The form $a$ induces a bounded linear operator $A: V\rightarrow V^*$ defined as follows:
\begin{equation*}
    (Av)(w)=\langle Av,w\rangle:=a(v,w),\quad v, w\in V.
\end{equation*}
Here $\langle\bullet,\bullet\rangle$ is the duality pairing between the primal and dual spaces.  We say $A: V\rightarrow V^*$ is Symmetric and Positive-Definite (SPD) if $\langle Av,v\rangle\geq0$ and $\langle Av,v\rangle=0\Longleftrightarrow v=0$ for any $v\in V$. Clearly $A$ induces an inner product and a norm on $V$:
\[
(v,w)_A=\langle Av,w\rangle,\quad \|v\|_A=\sqrt{(v,v)_A}.
\]
By $T^*: W_2^*\rightarrow W_1^*$ we denote the adjoint operator of $T: W_1\rightarrow W_2$.
In our analysis, the ambient space $V$ is either: 
\begin{itemize}
    \item a Sobolev space such as $H^1(\Omega)$, $H({\rm curl},\Omega)$, $H({\rm div},\Omega)$ on the continuous level;
    \item  a finite element space based on a finer grid $\mathcal{T}_{h/2}$ or higher order polynomials.
\end{itemize} 

Let $V_1, \ldots, V_N\subset V$ be $N$ subspaces of $V$. Let $I_h: V_h\hookrightarrow V$ and $I_k: V_k\hookrightarrow V$ be the natural inclusions. In practice, $V_k$ is a Sobolev or finite element space on a local sub-domain $\Omega_k$. The operator $A_k: V_k\rightarrow V_k^*$ is the restriction $A|_{V_k}$, i.e., 
\[
(A_kv_k)(w_k)=\langle A_kv_k,w_k\rangle=a(v_k,w_k),\quad v_k, w_k\in V_k.
\]

We adopt the notation $C_1\lesssim C_2$ provided $C_1\leq C_3C_2$ with $C_3$ being a generic constant independent of quantities of interest, e.g., the mesh size parameter $h$ and target functions. We say $C_1\eqsim C_2$ provided $C_1\lesssim C_2$, $C_2\lesssim C_1$. This notation generalizes to comparison between SPD operators in an obvious way. It has been observed in \cite{LiZikatanov2021CAMWA} that additive two-level preconditioning yields a reliable and efficient error estimator for SPD problems. For completeness, we include a brief proof. 
\begin{theorem}\label{thm:rSr}
Let $B_a=I_hA_h^{-1}I_h^*+S_a: V^*\rightarrow V$ with $S_a:=\sum_{k=1}^NI_kA_k^{-1}I_k^*$.
Assume that $B_a$ is a preconditioner for the SPD operator $A$, i.e.,
\begin{equation}\label{eq:Ba_bound}
c_0\langle R,B_aR\rangle\leq\langle R,A^{-1}R\rangle\leq c_1\langle R,B_aR\rangle,\quad\forall R\in V^*,
\end{equation}
where $c_0, c_1>0$ are uniform constants. Then with $r=f-Au_h$, it holds that  \begin{equation}\label{eq:errorBa}
\sqrt{c_0}\langle r,S_ar\rangle^\frac{1}{2}\leq\|u-u_h\|_A\leq\sqrt{c_1}\langle r,S_ar\rangle^\frac{1}{2}.
\end{equation}   
\end{theorem}
\begin{proof}
Using the relation $A(u-u_h)=r$, we have
\begin{equation*}
\|u-u_h\|_A=\langle A(u-u_h),u-u_h\rangle^\frac{1}{2}=\langle r,A^{-1}r\rangle^\frac{1}{2}.
\end{equation*}  
Noticing $\langle I_h^*r,v_h\rangle=\langle r,v_h\rangle=0~\forall v_h\in V_h$, we also obtain $B_ar=S_ar$.
Then \eqref{eq:errorBa} follows from the spectral equivalence assumption $A^{-1}\eqsim B_a$.
\end{proof}

If $V=V_h+\sum_{k=1}^NV_k$ is a stable subspace decomposition, i.e., for any $v \in V$, there exist $v_h\in V_h$ and $v_k\in V_k$ with $k=1, \ldots, N$ such that 
\begin{align*}
&v = v_h + \sum_{k=1}^N v_k,\\  &\|v_h\|_A^2 + \sum_{k=1}^N \|v_k\|_A^2 \lesssim \|v\|_A^2,    
\end{align*} 
then the spectral equivalence \eqref{eq:Ba_bound} holds and $B_a$ is a classical additive Schwarz preconditioner for $A$. The spaces $H^1(\Omega)$, $H({\rm curl},\Omega)$, $H({\rm div},\Omega)$ and their finite element discrete variants under consideration all admit an appropriate stable decomposition for the analysis. In the literature, $S_a$ is referred to as an additive smoother for damping the high frequency component of the iterative error. 
Theorem \ref{thm:rSr} implies that the smoother $S_a$ yields a reliable and efficient a posteriori error estimate of finite element method. The inverse $A_k^{-1}$ in $S_a$ often amounts to solving a local problem on $\Omega_k$. 

We remark that $\langle r,S_ar\rangle^\frac{1}{2}$ is applicable to SPD problems as well as symmetric but indefinite systems posed in a Cartesian product space $V_1\times V_2\times\cdots$, where $S_a={\rm diag}(S_{V_1},S_{V_2},\ldots)$ is block diagonal with each $S_{V_i}$ being an SPD smoother on $V_i$, see Section \ref{sect:saddle}. However, $\langle r,S_ar\rangle^\frac{1}{2}$ could not be applied to strongly non-symmetric  problems because $\langle r,S_ar\rangle$ is not necessarily nonnegative in such cases.
Our new observation is that a contractive iterative solver for $Au=f$
\[
u_{n+1}=u_n+B(f-Au_n),\quad n=0,1,2,\ldots
\]
yields an alternative a posteriori error estimate that may be applied to non-symmetric problems such as the Helmholtz and convection-diffusion equations. 
\begin{theorem}\label{thm:Brnorm}
Let $\|\bullet\|_M$ be a norm on $V$. Let $B: V^*\rightarrow V$ be a contractive iterator with respect to the $M$-norm, i.e.,
\[
\|I-BA\|_M<1.
\]
Then with $r=f-Au_h\in V^*$, we have
\begin{equation*}
    \frac{1}{2}\|Br\|_M<\|u-u_h\|_M\leq\frac{\|Br\|_M}{1-\|I-BA\|_M}.
\end{equation*}
\end{theorem}
\begin{proof}
We start with
\begin{equation}\label{error_Br}
\begin{aligned}
&\|u-u_h\|_M=\|A^{-1}B^{-1}Br\|_M\leq\|(BA)^{-1}\|_M\|Br\|_M,\\
&\|Br\|_M=\|BA(u-u_h)\|_M\leq\|BA\|_M\|u-u_h\|_M.
\end{aligned}
\end{equation}
The fact $\|I-BA\|_M<1$ implies that $BA$ is invertible and  
\begin{equation}\label{BAinvnorm}
 \|(BA)^{-1}\|_M\leq\frac{1}{1-\|I-BA\|_M}.   
\end{equation}
It also follows from $\|I-BA\|_M<1$ that 
\begin{equation}\label{BAnorm}
\|BA\|_M\leq1+\|I-BA\|_M<2.
\end{equation}
Combining \eqref{error_Br} with \eqref{BAinvnorm} and \eqref{BAnorm} completes the proof.
\end{proof}

\begin{remark}
In most cases, $\|\bullet\|_M$ is taken to be the underlying norm $\|\bullet\|_A$ induced by the SPD operator $A$. We note, however, that $\|\bullet\|_M$ practically need not coincide with the natural norm $\|\bullet\|_A$. For instance, choosing $\|\bullet\|_M = \|\bullet\|_{L^2(\Omega)}$ yields a posteriori error estimates in the $L^2$-norm for problems posed in $V = H^1(\Omega)$ or $V = H({\rm curl},\Omega)$; see Sections \ref{subsect:NE_Hgrad} and \ref{subsect:Hcurl}.
\end{remark}

Theorem \ref{thm:Brnorm} states that $\|Br\|_M$ is both an upper and lower bound of the FEM discretization error up to uniform multiplicative constants. In practice, $B$ often makes use of two-level or multi-level multiplicative structures to achieve uniform contraction. On the finer level, the multiplicative smoother $S_m$ given in Algorithm \ref{alg:SSC} is  widely used.  

\begin{algorithm}[thp]
\caption{Multiplicative smoother $S_m$}\label{alg:SSC}
\begin{algorithmic}
\STATE Let $u_0=0$ and $R\in V^*$;

\FOR{$k=1:N$}
    \STATE find $\eta_k\in V_k$ s.t. $a(\eta_k,v)=R(v)-a(u_{k-1},v),\quad v\in V_k$;
\STATE set $u_k=u_{k-1}+\eta_k$;
\ENDFOR

\STATE set $S_mR=u_N$.

\end{algorithmic}
\end{algorithm}
Here $S_m$ is formulated as a method of subspace correction (cf.~\cite{Xu1992,XuZikatanov2002}). When ${\rm dim}(V_k)=1$ for each $V_k$, Algorithm \ref{alg:SSC} reduces to a Gauss--Seidel iteration. It has been proved in \cite{XuZikatanov2002} that $\|I-S_mA\|_A<1$ provided $A$ is SPD and $V=\sum_{k=1}^NV_k$. In general, $\|I-S_mA\|_A$ is not uniformly bounded away from 1 and a coarse level correction in $V_h$ is needed for uniform convergence.
Similarly to Theorem \ref{thm:rSr}, a key observation is that the \emph{smoother} in a two-level iterator $B$ (rather than $B$ itself) suffices to provide reliable and efficient a posteriori error estimates. 
\begin{corollary}\label{cor:Srnorm}
Let $\|\bullet\|_M$ be a norm on $V$ and $S_m: V^*\rightarrow V$ be given in Algorithm \ref{alg:SSC}. Assume that $V=V_h+\sum_{k=1}^NV_k$. Then the two-level iteration
\begin{subequations}\label{alg:twol_successive}
\begin{align}
\eta&=u_n+A_h^{-1}I_h^*(f-Au_n),\\
u_{n+1}&= \eta+S_m(f-A\eta),\quad n=0, 1, 2, \ldots
\end{align}    
\end{subequations}
for solving $Au=f$
is contractive, i.e., $\|I-B_mA\|_M<1$ with  $B_m:=S_m+A_h^{-1}I_h^*-S_mAA_h^{-1}I_h^*$. In addition, we have 
\begin{equation*}
    \frac{1}{2}\|S_mr\|_M<\|u-u_h\|_M\leq\frac{\|S_mr\|_M}{1-\|I-B_mA\|_M}.
\end{equation*}
\end{corollary}
\begin{proof}
The contraction of \eqref{alg:twol_successive} follows from $V=V_h+\sum_{k=1}^NV_k$ and the classical result in \cite{XuZikatanov2002}. The iterative error of \eqref{alg:twol_successive} satisfies 
\begin{equation*}
u-u_{n+1}=(I-S_mA)(I-A_h^{-1}I_h^*A)(u-u_n)=(I-B_mA)(u-u_n).    
\end{equation*}
Then contraction property of the algorithm \eqref{alg:twol_successive} translates into $\|I-B_mA\|_M<1$. Due to $V_h\subset V$ and the definition of the Galerkin discretization \eqref{subeq:FEM}, we have $I_h^*r=0$ and thus $B_mr=S_mr.$ Using this identity and Theorem \ref{thm:Brnorm} completes the proof. 
\end{proof}
When $A$ is the SPD operator induced by the inner product of $H^1(\Omega)$, $H({\rm curl},\Omega)$, $H({\rm div},\Omega)$ or their finite element variants, the iteration \eqref{alg:twol_successive} with $S_m$ given by a stable decomposition is uniformly contractive under the norm $\|\bullet\|_A$, i.e., $\|I-B_mA\|_A<1$ is independent of the mesh size parameter $h$.

\subsection{Application to the Poisson Equation}\label{subsect:Poisson}
On a domain $\Omega\subset\mathbb{R}^d$, consider the Poisson boundary value problem 
\begin{equation}\label{eq:Poisson}
    \begin{aligned}
        -\Delta u&=f\quad\text{ in }\Omega,\\
        u&=0\quad\text{ on }\partial\Omega.
    \end{aligned}
\end{equation}
Let $(\bullet,\bullet)_D$ be the $L^2$ inner product on $D$ and $(\bullet,\bullet)=(\bullet,\bullet)_\Omega$. Let $\mathcal{P}_p$ be the space of polynomials of degree $\leq p$.
A typical application of Theorem \ref{thm:rSr} and Corollary \ref{cor:Srnorm} is implicit a posteriori error estimation of FEMs for \eqref{eq:Poisson}: $u\in H_0^1(\Omega)$ and $u_h\in V_h$ such that
\begin{subequations}\label{eq:varPoisson}
\begin{align}
a(u,v)=(\nabla u,\nabla v)&=(f,v),\quad v\in H_0^1(\Omega),\\
(\nabla u_h,\nabla v_h)&=(f,v_h),\quad v_h\in V_h,\label{eq:varPoisson_FEM}
\end{align}
\end{subequations}
where $V_h\subset H_0^1(\Omega)$ is a finite element space of continuous and piecewise $\mathcal{P}_1$ polynomials.
Let $\{a_{h,k}\}_{1\leq k\leq N_h}$ be the set of grid vertices in $\mathcal{T}_h$ and $\phi_{h,k}\in V_h$ the hat function associated to $a_{h,k}$. Then  $\Omega=\cup_{k=1}^{N_h}\Omega_{h,k}$ with $\Omega_{h,k}={\rm supp}(\phi_{h,k})$ is an overlapping domain decomposition and $H_0^1(\Omega)=V_h+\sum_{k=1}^{N_h}H_0^1(\Omega_{h,k})$ is a stable subspace decomposition {(cf.~\cite{LiZikatanov2021CAMWA}, Lemma 3.3)}. In this case, the estimator $S_a$ in Theorem \ref{thm:rSr} leads to 
\begin{equation}\label{eq:Poisson_localDiri_a}
|u-u_h|_{H^1(\Omega)}\eqsim\langle r,S_ar\rangle^\frac{1}{2}=\Big(\sum_{k=1}^{N_h}|\eta_k|_{H^1(\Omega_{h,k})}^2\Big)^\frac{1}{2},  
\end{equation}
where each $\eta_k\in H_0^1(\Omega_{h,k})$ is determined by the local Dirichlet problem
\[
(\nabla\eta_k,\nabla v)_{\Omega_{h,k}}=(f,v)_{\Omega_{h,k}}-(\nabla u_h,\nabla v)_{\Omega_{h,k}}
,\quad v\in H_0^1(\Omega_{h,k}).
\]

Let $S_m: H^{-1}(\Omega)\rightarrow H^1_0(\Omega)$ be a multiplicative smoother based on $V_k=H_0^1(\Omega_{h,k})$, $k=1, \ldots, N_h$, see Algorithm \ref{alg:SSC} with $V^*=H^{-1}(\Omega)$ and $V_k=H_0^1(\Omega_{h,k})$. Classical multigrid theory (cf.~\cite{Xu1992,XuZikatanov2017}) implies that $B_m$ in Corollary \ref{cor:Srnorm} is a uniform contraction, i.e., $|I-B_mA|_{H^1(\Omega)}$ ($A=-\Delta$) is uniformly bounded away from 1. Consequently, we obtain a multiplicative a posteriori error estimate
\begin{equation}\label{eq:Poisson_localDiri_m}
|u-u_h|_{H^1(\Omega)}\eqsim|S_mr|_{H^1(\Omega)}.    
\end{equation}
A posteriori error estimates \eqref{eq:Poisson_localDiri_a} and \eqref{eq:Poisson_localDiri_m} require solutions of infinite-dimensional local problems that can only be approximately solved. Let $$\widetilde{V}_k=\{v\in H_0^1(\Omega_{h,k}): v|_T\in\mathcal{P}_2\text{ for each }T\subset\Omega_{h,k}\text{ with }T\in\mathcal{T}_h\}$$ be the local $\mathcal{P}_2$ finite element space. 
One could approximately solve for each $\eta_k$ in \eqref{eq:Poisson_localDiri_a} by $\mathcal{P}_2$ element: find $\eta_k\in\widetilde{V}_k$ such that
\[
(\nabla\eta_k,\nabla v)_{\Omega_{h,k}}=(f,v)_{\Omega_{h,k}}-(\nabla u_h,\nabla v)_{\Omega_{h,k}}
,\quad v\in\widetilde{V}_k.
\]
The estimator \eqref{eq:Poisson_localDiri_m} could be implemented in a similar way. 

\section{Discrete Ambient Space}\label{sect:discreteV} 
To directly construct more practical and efficient a posteriori error estimates, we make use of finite-dimensional discrete ambient spaces $V$ arising from mesh refinement or polynomial enrichment. 

\subsection{Practical Estimators in H(grad)}\label{subsect:Hgrad}

Let $\mathcal{T}_{h/2}$ be a uniform refinement of $\mathcal{T}_h$ and $V=V_{h/2}$ be the finite element space based on the refined mesh $\mathcal{T}_{h/2}$. For simplicial meshes, $\mathcal{T}_{h/2}$ is constructed by dividing each element in $\mathcal{T}_h$ into four congruent children in $\mathbb{R}^2$ or by dividing each element in $\mathcal{T}_h$ into eight children (cf.~\cite{Bey2000}) in $\mathbb{R}^3$. 

\subsubsection{Pointwise Smoother-type Error Estimators}
Let $A_{h/2}: V_{h/2}\rightarrow V_{h/2}^*$ be given by
\begin{equation*}
    (A_{h/2}v)(w)=\langle A_{h/2}v,w\rangle=a(v,w),\quad v, w\in V_{h/2}.
\end{equation*}
Let $\mathbb{A}_{h/2}$ be the stiffness matrix of the bilinear form $a|_{V_{h/2}}$ and $S^a_{h/2}: V_{h/2}^*\rightarrow V_{h/2}$ be the operator corresponding to the inverse of the diagonal of $\mathbb{A}_{h/2}$. Corresponding to the stable decomposition $V_{h/2}=V_h+\sum_{a_{h/2,k}\in\mathring{\Omega}}{\rm span}(\phi_{h/2,k})$, the following
\[
B_{h/2}^a:=I_hA_h^{-1}I_h^*+S^a_{h/2}
\]
is a well-known uniform preconditioner for $A_{h/2}$, where $I_h: V_h\hookrightarrow V_{h/2}$ is the natural inclusion or coarse-to-fine prolongation by slight abuse of notation. Let $S^m_{h/2}: V_{h/2}^*\rightarrow V_{h/2}$ be the operator corresponding to the inverse of the upper triangular part of $\mathbb{A}_{h/2}$. Another well-known two-level multiplicative solver for $A_{h/2}u_{h/2}=f_{h/2}$ is  
\begin{align*}
\eta&=u_n+A_h^{-1}I_h^*(f-A_{h/2}u_n),\\
u_{n+1}&= \eta+S^m_{h/2}(f-A_{h/2}\eta),\quad n=0, 1, 2, \ldots
\end{align*}    
In the above iteration, $u_n$ converges to $u_{h/2}$ uniformly with respect to $h$ (cf.~\cite{XuZikatanov2017}). On the matrix level, $S^a_{h/2}$ and $S^m_{h/2}$ are nothing but one step of Jacobi and Gauss--Seidel smoothing pass for $\mathbb{A}_{h/2}$, respectively. On the analytic level, $S^a_{h/2}$ and $S^m_{h/2}$ amount to solving one-dimensional local problems and are referred to as pointwise smoothers. 

Let $r_{h/2}:=f-A_{h/2}u_h\in V_{h/2}^*$, namely,
\[
\langle r_{h/2},v\rangle=(f,v)-a(u_h,v),\quad\forall v\in V_{h/2}.
\]
It follows from Theorem \ref{thm:rSr} and Corollary \ref{cor:Srnorm} that 
\begin{subequations}\label{eq:uh_uh2}
 \begin{align}
    |u_{h/2}-u_h|_{H^1(\Omega)}&\eqsim\langle r_{h/2},S^a_{h/2}r_{h/2}\rangle^\frac{1}{2},\\
    |u_{h/2}-u_h|_{H^1(\Omega)}&\eqsim|S^m_{h/2}r_{h/2}|_{H^1(\Omega)}.
\end{align}   
\end{subequations}

However, this section is devoted to a posteriori error estimation of $|u-u_h|_{H^1(\Omega)}$ other than $|u_{h/2}-u_h|_{H^1(\Omega)}$. To achieve our goal, we make use of the saturation assumption on fine-coarse meshes:
\begin{equation}\label{eq:saturation}
    \|u-u_{h/2}\|_V\leq\gamma_{h/2}\|u-u_h\|_V
\end{equation}
for some uniform constant $\gamma_{h/2}<1$. Here $\|\bullet\|_V=|\bullet|_{H^1(\Omega)}$ but $\|\bullet\|_V$ could be other norms in the following sections. Assumption \eqref{eq:saturation} is common in error analysis of FEMs (cf.~\cite{BankWeiser1985,Bank1996,BankOvall2017}) and has been proved for the Poisson equation in, e.g., \cite{DorflerNochetto2002,CarstensenGallistlGedicke2016}. It follows from the triangle inequality
\begin{equation*}
\begin{aligned}
&|u-u_h|_{H^1(\Omega)}\leq|u-u_{h/2}|_{H^1(\Omega)}+|u_{h/2}-u_h|_{H^1(\Omega)}\\  
&    |u-u_h|_{H^1(\Omega)}\geq|u_{h/2}-u_h|_{H^1(\Omega)}-|u-u_{h/2}|_{H^1(\Omega)}
\end{aligned}
\end{equation*}
and assumption \eqref{eq:saturation} that 
\begin{equation}\label{eq:error_uhuh2}
    \frac{1}{1+\gamma_{h/2}}|u_{h/2}-u_h|_{H^1(\Omega)}\leq|u-u_h|_{H^1(\Omega)}\leq\frac{1}{1-\gamma_{h/2}}|u_{h/2}-u_h|_{H^1(\Omega)}.
\end{equation}
Combining \eqref{eq:uh_uh2} and \eqref{eq:error_uhuh2}, we obtain 
\begin{subequations}\label{eq:Poisson_finecoarse}
    \begin{align}
    |u-u_h|_{H^1(\Omega)}&\eqsim\langle r_{h/2},S^a_{h/2}r_{h/2}\rangle^\frac{1}{2},\label{eq:Poisson_finecoarse_a}\\
    |u-u_h|_{H^1(\Omega)}&\eqsim|S^m_{h/2}r_{h/2}|_{H^1(\Omega)}.
\end{align}
\end{subequations}

Alternatively, we can construct smoothers based on polynomial enrichment.
Replacing $V=V_{h/2}$ with the finite element space $V=V_{h,\mathcal{P}_2}$ based on piecewise quadratic polynomials on $\mathcal{T}_h$, we obtain another smoother-type a posteriori error estimate. Let  $S_{h,\mathcal{P}_2}^a: V_{h,\mathcal{P}_2}^*\rightarrow V_{h,\mathcal{P}_2}$ and $S^m_{h,\mathcal{P}_2}: V_{h,\mathcal{P}_2}^*\rightarrow V_{h,\mathcal{P}_2}$ be the operator corresponding to the inverse of the diagonal and the inverse of the upper triangular part of the $\mathcal{P}_2$ finite element matrix on $\mathcal{T}_h$, respectively. Let $u_{h,\mathcal{P}_2}$ be the finite element solution in $V_{h,\mathcal{P}_2}$ and $A_{h,\mathcal{P}_2}: V_{h,\mathcal{P}_2}\rightarrow V_{h,\mathcal{P}_2}^*$ be given by
\begin{equation*}
    (A_{h,\mathcal{P}_2}v)(w)=\langle A_{h,\mathcal{P}_2}v,w\rangle=a(v,w),\quad v, w\in V_{h,\mathcal{P}_2}.
\end{equation*}
Under the following $p$-version saturation assumption (cf.~\cite{CanutoNochettoStevensonVerani2019})
\begin{equation}\label{eq:saturation_P2}
    \|u-u_{h,\mathcal{P}_2}\|_V\leq\gamma_{\mathcal{P}_2}\|u-u_h\|_V,\quad\gamma_{\mathcal{P}_2}<1,
\end{equation}
we obtain smoother-type error estimators based on a low-high polynomial degree finite element pair:
\begin{subequations}\label{eq:Poisson_P1P2}
\begin{align}
    |u-u_h|_{H^1(\Omega)}&\eqsim\langle r_{\mathcal{P}_2},S^a_{h,\mathcal{P}_2}r_{\mathcal{P}_2}\rangle^\frac{1}{2},\label{eq:Poisson_P1P2_a}\\
    |u-u_h|_{H^1(\Omega)}&\eqsim|S^m_{h,\mathcal{P}_2}r_{\mathcal{P}_2}|_{H^1(\Omega)},
\end{align}
\end{subequations}
where $r_{\mathcal{P}_2}=f-A_{h,\mathcal{P}_2}u_h\in V_{h,\mathcal{P}_2}^*$.

\begin{remark}\label{rem:extrapolation}
The implicit extrapolation technique \cite{jung1996implicit} provides an efficient way to assemble a $\mathcal{P}_2$ smoother matrix $\mathbb{S}_{h,\mathcal{P}_2}$ via $\mathbb{S}_{h,\mathcal{P}_2}=(4/3)\mathbb{S}_{h/2}- {(1/3)\tilde{\mathbb{S}}_{h}}$, where $\mathbb{S}_{h/2}$ and $\mathbb{S}_{h}$ represent the $\mathcal{P}_1$ smoothers $S_{h/2}$ and $S_h$, {and $\tilde{\mathbb{S}}_{h}$ is the embedding of $\mathbb{S}_h$ into the fine grid matrix size, with zeros assigned to the nodes present in the fine grid $\mathcal{T}_{h/2}$ but not in the coarse grid $\mathcal{T}_h$.}
\end{remark}

\begin{figure}[!th]
    \centering
    \includegraphics[width=7.5cm]{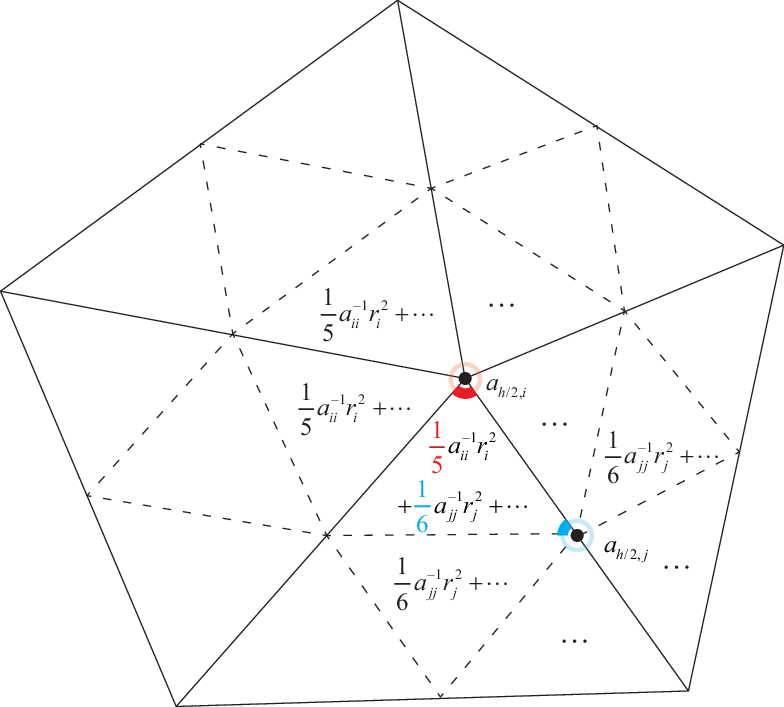} 
\caption{Assignment of $\langle r_{h/2},S^a_{h/2}r_{h/2}\rangle$ to each element in the refinement $\mathcal{T}_{h/2}$.}\label{fig:patch}
\end{figure}

\subsubsection{Localized Error Estimator and Adaptive Algorithm}
Let $\mathcal{V}_{h/2}$ be the set of interior vertices in $\mathcal{T}_{h/2}$. Let $\mathbf{r}_{h/2}=(r_i)_{1\leq i\leq M}$ be the coordinate vector of $r_{h/2}$ with $M=\#\mathcal{V}_{h/2}$ and $\mathbb{D}_{h/2}={\rm diag}(\mathbb{A}_{h/2})=(a_{ii})_{1\leq i\leq M}$. Then the algebraic expression of the estimator in \eqref{eq:Poisson_finecoarse_a} is 
\[
\langle r_{h/2},S^a_{h/2}r_{h/2}\rangle=\mathbf{r}^\top\mathbb{D}_{h/2}^{-1}\mathbf{r}=\sum_{a_{h/2,i}\in\mathcal{V}_{h/2}}a_{ii}^{-1}r_i^2.\]
To guide the local mesh refinement, each $a_{ii}^{-1}r_i^2$ in 
$\langle r_{h/2},S^a_{h/2}r_{h/2}\rangle$ is equally subdivided {(according to the vertex degree ${\rm deg}(a_{h/2,i})$, i.e., the number of edges in $\mathcal{T}_{h/2}$ that are incident to $a_{h/2,i}$)} and allocated to all elements in $\mathcal{T}_{h/2}$ touching the vertex $a_{h/2,i}\in\mathcal{V}_{h/2}$, see Figure \ref{fig:patch} for example. The localized estimator for each $T\in\mathcal{T}_h$ reads
\begin{equation*}
\mathcal{E}_T=\left(\sum_{t\subset T, t\in\mathcal{T}_{h/2}}\sum_{\substack{a_{h/2,i}\in\bar{t}\\ a_{h/2,i}\in\mathcal{V}_{h/2}}}\frac{1}{{\rm deg}(a_{h/2,i})}a_{ii}^{-1}r_{i}^2\right)^\frac{1}{2}.
\end{equation*}
The error estimator in \eqref{eq:Poisson_P1P2_a} is localized to each triangle of $\mathcal{T}_h$ in a similar way. Naturally, $|S^m_{h/2}r_{h/2}|_{H^1(\Omega)}$ and $|S^m_{h,\mathcal{P}_2}r_{\mathcal{P}_2}|_{H^1(\Omega)}$ are localized based on the following splitting:
\begin{align*}
|S^m_{h/2}r_{h/2}|_{H^1(\Omega)}&=\big(\sum_{T\in\mathcal{T}_h}|S^m_{h/2}r_{h/2}|_{H^1(T)}^2\big)^\frac{1}{2},\\ |S^m_{h,\mathcal{P}_2}r_{\mathcal{P}_2}|_{H^1(\Omega)}&=\big(\sum_{T\in\mathcal{T}_h}|S^m_{h,\mathcal{P}_2}r_{\mathcal{P}_2}|_{H^1(T)}^2\big)^\frac{1}{2}.    
\end{align*}

\begin{algorithm}[!ht]
\caption{AFEM driven by Smoother-type Error Estimator}
    \label{alg:smoother_afem}
        \begin{algorithmic}[1] 
            \STATE \textbf{Input:} current mesh $\mathcal{T}_h$;
            
            \textbf{// 1. SOLVE}
            \STATE Solve the finite element system on $\mathcal{T}_h$ to obtain $u_h$; \\
            
            \textbf{// 2. ESTIMATE} 
            \STATE Construct auxiliary fine mesh $\mathcal{T}_{h/2}$;
            \STATE Compute fine-grid problem-specific residual $r_{h/2}\in V_{h/2}^*$: $\langle r_{h/2}, v \rangle = (f, v) - a(u_h, v), \quad \forall v \in V_{h/2}$;
            \STATE Assemble the smoother $S_{h/2}$ and compute $\langle r_{h/2},S_{h/2}r_{h/2}\rangle$ or $\|S_{h/2}r_{h/2}\|$;\\
            
            \textbf{// 3. ADAPT}
            \STATE Localize $\langle r_{h/2},S_{h/2}r_{h/2}\rangle$ or $\|S_{h/2}r_{h/2}\|$ to obtain $\{\mathcal{E}_T\}_{T \in \mathcal{T}_{h}}$;
            \STATE Mark a $\mathcal{M}_h \subset \mathcal{T}_h$ {with minimal cardinality such that} $\sum_{T \in \mathcal{M}_h} \mathcal{E}_T^2 \ge \theta^2 \sum_{T \in \mathcal{T}_h} \mathcal{E}_T^2$,~~  $\theta\in(0,1)$;
            \STATE Update $\mathcal{T}_h$ by refining $\mathcal{M}_h$ and neighboring elements based on newest vertex bisection, go to Step 1;
        \end{algorithmic}
    \end{algorithm}

Algorithm \ref{alg:smoother_afem} outlines an adaptive mesh refinement procedure driven by smoother-type a posteriori error estimates. It is characterized by several inexpensive operations on an auxiliary fine mesh, which avoid solving any global system. In particular, the residual evaluation, construction of smoothers, and smoothing of the residual all exhibit linear complexity. The adaptive feedback loop follows the classical AFEM framework (cf.~\cite{Dorfler1996,CKNS2008,BonitoCanutoNochettoVeeser2024}) based on D\"orfler's marking and newest vertex bisection rule. 

In all numerical experiments, the marking parameter $\theta=0.5$ in Algorithm \ref{alg:smoother_afem} will be used. Although the PDE, the domain, the finite element space $V_h$, the residual $r_h \in V_h^*$, and the estimator $\mathcal{E}_T$ may vary from case to case, we adhere to the general adaptive workflow outlined in Algorithm \ref{alg:smoother_afem}. Numerical experiments were conducted in MATLAB 2024b, with linear systems solved by the direct solver {\texttt{mldivide}}. 
The quantity shown in parentheses in each figure legend is the effectivity ratio, i.e., the value of estimator/error {on the finest mesh generated by the final AFEM iteration}.

\subsubsection{Numerical Experiments in H(grad)}\label{subsect:NE_Hgrad}
On the L-shaped domain $\Omega=(-1,1)^2\backslash([0,1)\times[-1,0))$, we test the performance of the proposed a posteriori error estimates for the FEM \eqref{eq:varPoisson_FEM}. The exact solution is $u=r^{2/3}\sin(2\theta/3)$ with singularity near the re-entrant corner of $\Omega$, where $(r,\theta)$ is the polar coordinate near the origin. All quantities are computed on adaptively generated meshes by Algorithm \ref{alg:smoother_afem}. The smoother-type a posteriori error estimates as well as computable discrete approximations of implicit estimators \eqref{eq:Poisson_localDiri_a} and \eqref{eq:Poisson_localDiri_m} based on local $\mathcal{P}_2$ resolution are compared to the explicit residual error estimator 
\begin{equation}\label{eq:residualEstimate}
    \mathcal{E}_{\rm res}=\Big(\sum_{T\in\mathcal{T}_h}h_T^2\|f+\Delta u_h\|_{L^2(T)}^2+\sum_{\text{edge }E\subset\mathring{\Omega}}h_E\|\llbracket\partial_nu_h\rrbracket\|^2_{L^2(E)}\Big)^\frac{1}{2},
\end{equation}
where $h_T={\rm diam}(T)$, $h_E={\rm diam}(E)$, and $\llbracket\partial_nu_h\rrbracket|_E$ is the jump of the normal derivative $\partial_nu_h$ across an edge $E$.

\begin{figure}[!ht]
    \centering
    \includegraphics[width=5.5cm]{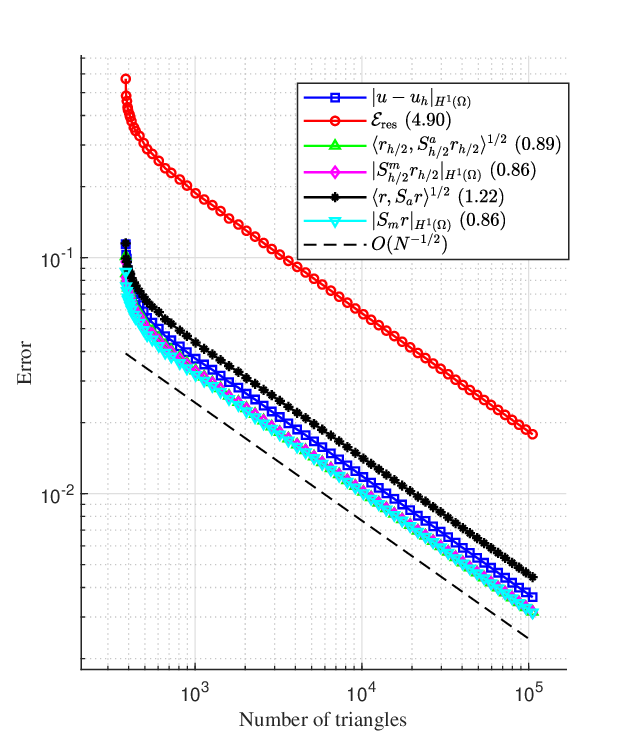} 
     \includegraphics[width=5.5cm]{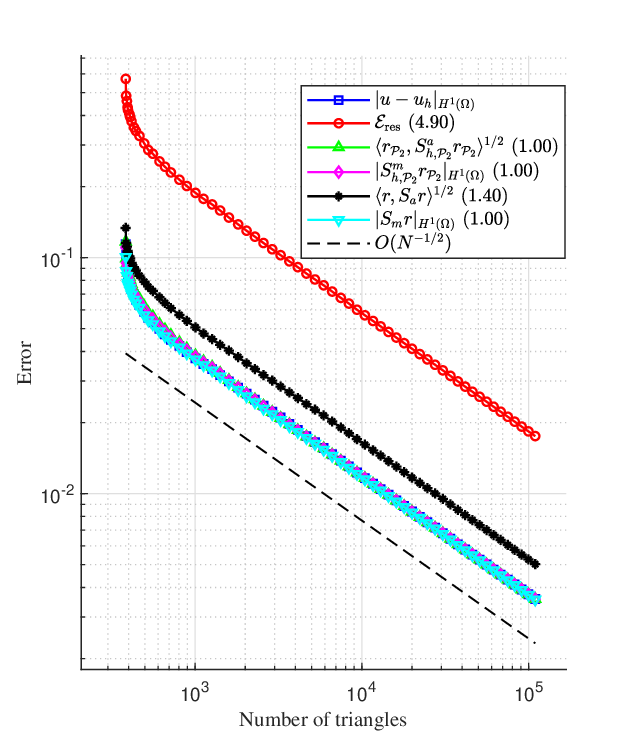}
\caption{Convergence of error estimators based on fine-coarse (left) and $\mathcal{P}_2$-$\mathcal{P}_1$ (right) layers.}\label{fig:Poisson}
\end{figure}

\begin{figure}[!ht]
        \centering
        \includegraphics[width=5.5cm]{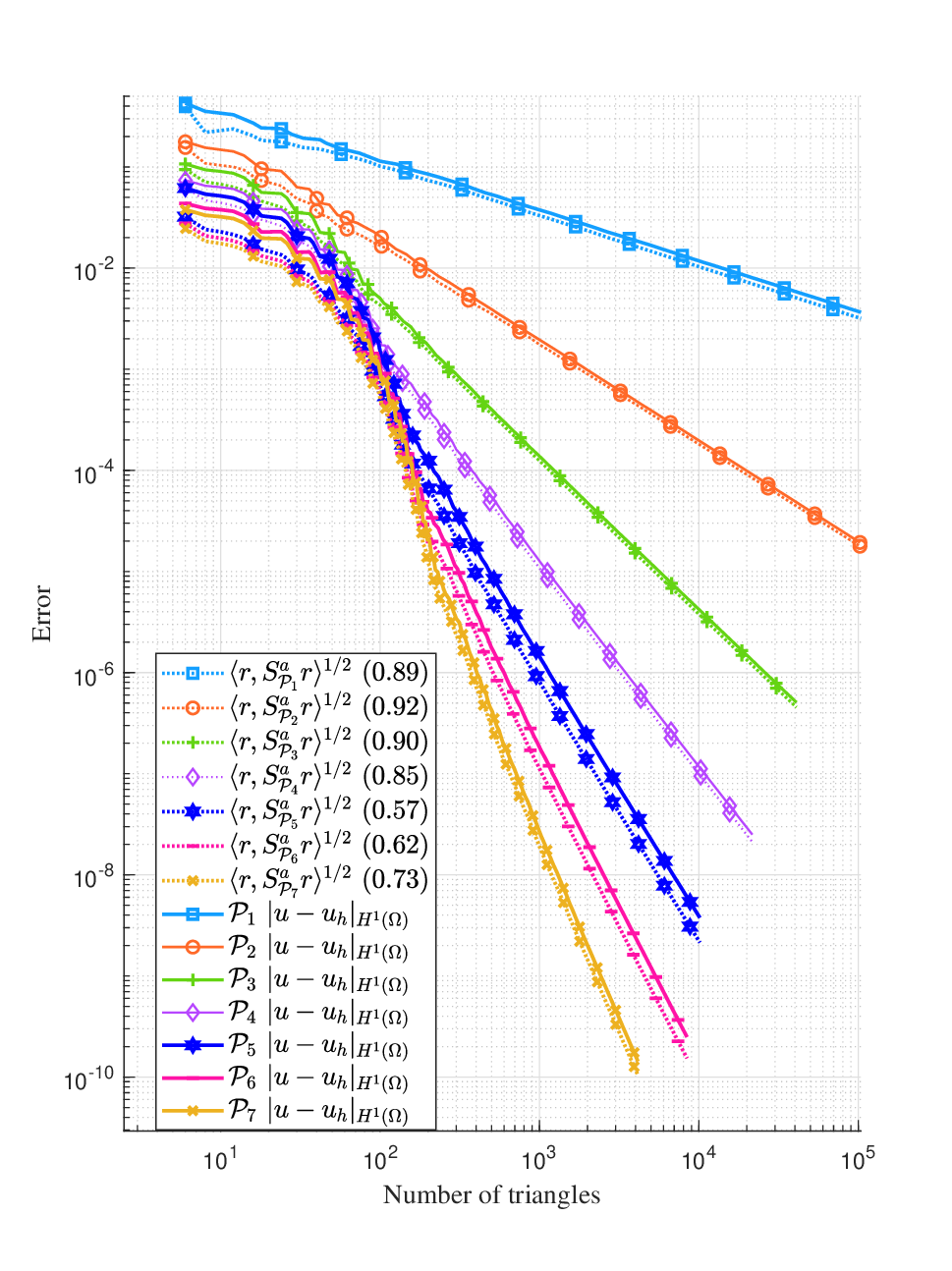} 
        \includegraphics[width=5.5cm]{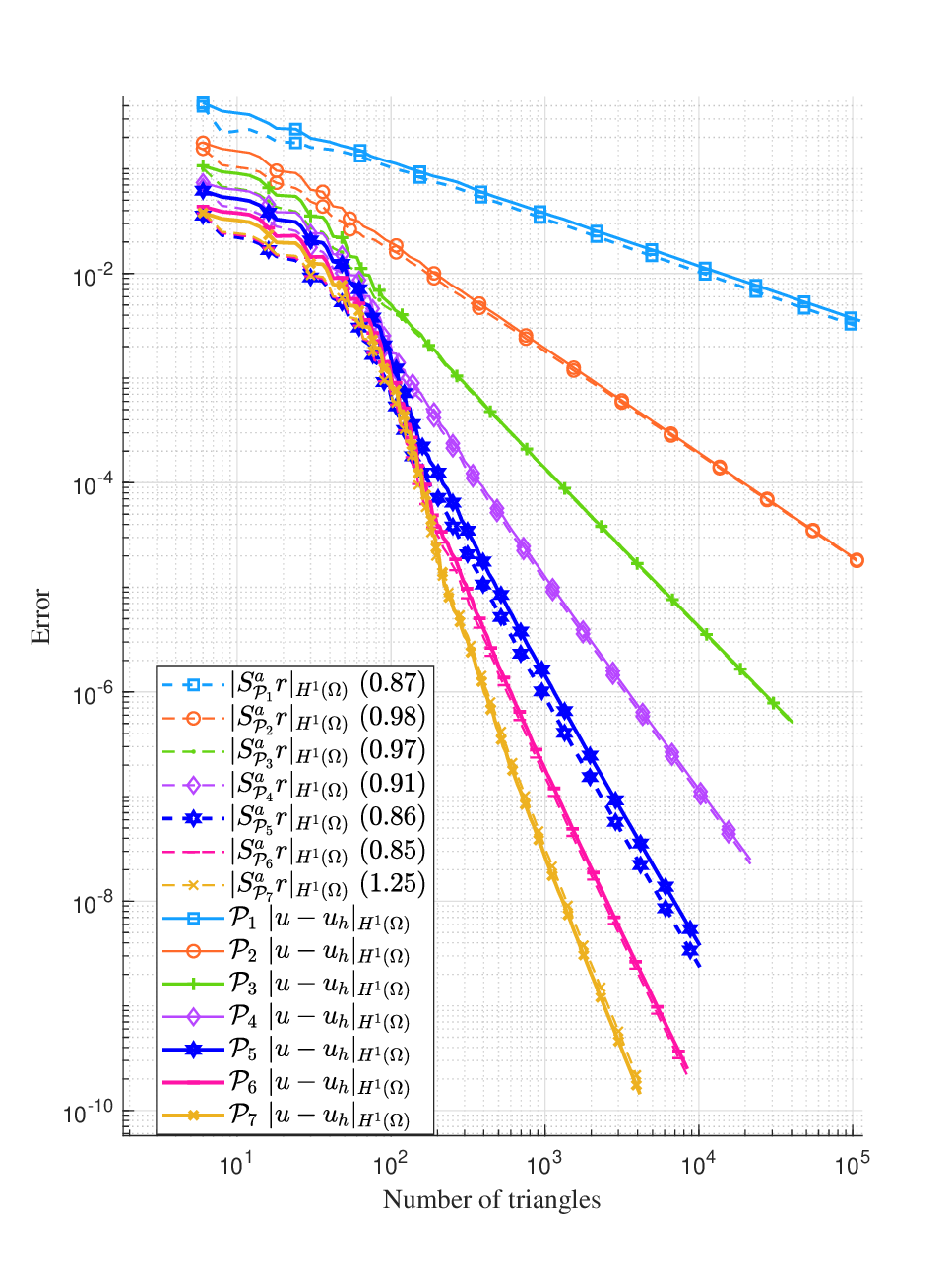}
    \caption{Convergence of $\mathcal{P}_p$-AFEMs driven by pointwise smoother-type error estimators.}
    \label{fig:Poisson_Pk}
\end{figure}

\begin{remark}[Inhomogeneous essential boundary condition in H(grad)]
For inhomogeneous and piecewise polynomial boundary data $u|_{\partial\Omega}=g_h\in V_h|_{\partial\Omega}$, the finite element solution $u_h\in V_h+g_h$ exactly matches $u$ and $u_{h/2}$ on the boundary. As a result, $u-u_h\in H_0^1(\Omega)={\rm Dom}(-\Delta)$ and $u_{h/2}-u_h\in V_{h/2}={\rm Dom}(A_{h/2})$ are contained in the domain of $-\Delta: H_0^1(\Omega)\rightarrow H^{-1}(\Omega)$ and $A_{h/2}$, respectively. In this case, the same smoothers for $-\Delta$ and $A_{h/2}$ in Sections \ref{subsect:Poisson} and \ref{subsect:Hgrad} are applied to the residual to obtain smoother-type a posteriori error estimates. For general boundary data $u|_{\partial\Omega}=g$, {we implement the discrete boundary condition $u_h|_{\partial\Omega}=g_h$ with $g_h\in V_h|_{\partial\Omega}$ being the nodal interpolant of $g$. Rigorously speaking, the data oscillation ${\rm osc}_\partial=\big(\sum_{{\rm edge}~E\subset\partial\Omega}h_E|g-g_h|_{H^1(E)}^2\big)^\frac{1}{2}$ for 2D problems should be incorporated into both smoother-type and other classical error estimators to obtain rigorous a posteriori error upper bounds (cf.~\cite{MNS2003,FeischlPagePraetorius2014}). In practice, the term ${\rm osc}_\partial$ is of higher order than the FEM error and therefore does not significantly influence the performance of AFEMs. For simplicity of presentation, we omit ${\rm osc}_\partial$ from all error estimators under comparison in the subsequent experiments.}
\end{remark}

We plot error curves for various AFEMs driven by different error estimators to assess convergence.
It is observed from Figure \ref{fig:Poisson} that the smoother-type estimators are more accurate than $\mathcal{E}_{\rm res}$. Meanwhile, they are no worse than the more expensive error estimators \eqref{eq:Poisson_localDiri_a} and \eqref{eq:Poisson_localDiri_m} based on local Dirichlet problems on vertex patches. In the comparison, the most efficient a posteriori estimates are based on diagonal smoothers $S_{h/2}^a$ and $S_{h,\mathcal{P}_2}^a$.

Smoother-type a posteriori error estimates could be directly generalized to higher order finite elements. Let  $S_{\mathcal{P}_p}^a: V_{\mathcal{P}_p}^*\rightarrow V_{\mathcal{P}_p}$ and $S^m_{\mathcal{P}_p}: V_{\mathcal{P}_p}^*\rightarrow V_{\mathcal{P}_p}$ be the operators corresponding to the inverse of the diagonal and the inverse of the upper triangular part of the $\mathcal{P}_p$ finite element matrix on $\mathcal{T}_{h/2}$, respectively. With $r=r_{h/2}$, the polynomial-degree-robustness ($p$-robustness) of smoother-type estimators $\langle r,S_{\mathcal{P}_p}^ar\rangle^\frac{1}{2}$ and  $|S_{\mathcal{P}_p}^ar|_{H^1(\Omega)}$ is numerically demonstrated in Figure \ref{fig:Poisson_Pk} and Table \ref{tab:estimator_compare}. In addition, Table \ref{tab:estimator_compare} reports effectivity ratios of the residual estimator $\mathcal{E}_{\rm res}$ and the equilibrated estimator $\mathcal{E}_{\rm eq}$ in \cite{ErnVohralik2015} for controlling the $\mathcal{P}_p$ finite element errors with $1\leq p\leq7$. The numerical results confirm that $\mathcal{E}_{\rm res}$ is not $p$-robust: $\mathcal{E}_{\rm res}/|u-u_h|_{H^1(\Omega)}$ grows at least linearly as $p$ increases. In contrast, $\mathcal{E}_{\mathrm{eq}}$ provides sharp effectivity ratios $\mathcal{E}_{\rm eq}/|u-u_h|_{H^1(\Omega)}$ close to 1, albeit at the cost of a more involved, implicit patchwise flux reconstruction.

Inspired by multigrid methods that perform multiple smoothing on each level, we consider a smoother-type error estimator $|S^{m,\text{iter}}_{\mathcal{P}_p}r|_{H^1(\Omega)}$, which corresponds to applying iter steps of Gauss–Seidel smoothing to the residual. Table \ref{tab:smoother_iter} shows that raising iter from 1 to 4 yields only a modest gain in the effectivity ratio. Therefore, we adopt a single smoothing pass in all subsequent experiments.

Compared to $\langle r_{h/2},S_{h/2}^ar_{h/2}\rangle^\frac{1}{2}$, we note that $|S_{h/2}^ar_{h/2}|_{H^1(\Omega)}$ is more flexible as it is easily adapted to other norms. For instance, the smoother-type estimator $\|S_{h/2}^ar_{h/2}\|_{L^2(\Omega)}$ is able to guide the adaptive FEM for reducing the $L^2$ error $\|u-u_h\|_{L^2(\Omega)}$, see Figure~\ref{fig:PoissonFineCoarseL2}. Moreover, the a posteriori error bound $\|S_{h/2}^ar_{h/2}\|_{L^2(\Omega)}$ is sharper than the following  residual error estimator 
\[
\mathcal{E}_{{\rm res},L^2}=\Big(\sum_{T\in\mathcal{T}_h}h_T^4\|f+\Delta u_h\|_{L^2(T)}^2+\sum_{\text{edge }E\subset\mathring{\Omega}}h_E^3\|\llbracket\partial_nu_h\rrbracket\|^2_{L^2(E)}\Big)^\frac{1}{2}
\]
for controlling the $L^2$ norm error (cf.~\cite{DemlowStevenson2011}).

\begin{figure}[!ht]
\centering
\includegraphics[width=5.5cm]{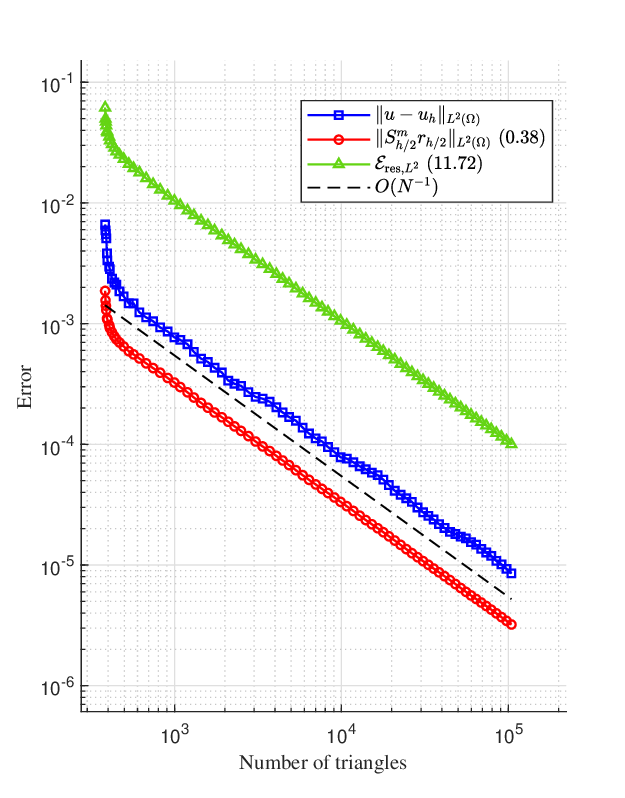} 
\includegraphics[width=5.5cm]{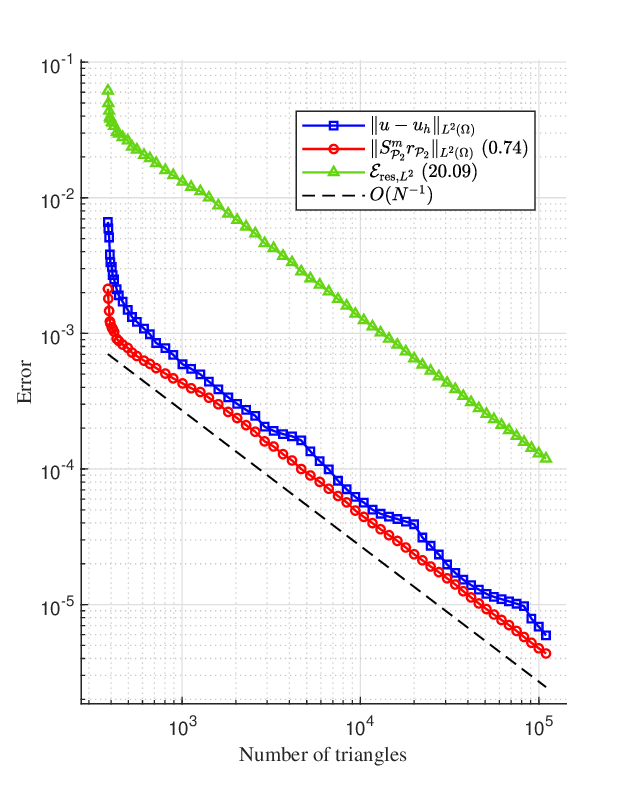}
\caption{Adaptive error estimator convergence in the $L^2$ norm for the Poisson equation in fine-coarse layers (left) and $\mathcal{P}_2$-$\mathcal{P}_1$ layers (right), where $N$ is the number of triangles.}
\label{fig:PoissonFineCoarseL2}
\end{figure}

\begin{table}[htbp]
  \centering
{\caption{Effectivity ratios of smoother-type $\langle r, S^a_{\mathcal{P}_p} r \rangle^{1/2}$, $|S^{a}_{\mathcal{P}_p}r|_{H^1(\Omega)}$ and $|S^{m}_{\mathcal{P}_p}r|_{H^1(\Omega)}$, residual $\mathcal{E}_{\rm{res}}$, and equilibrated $\mathcal{E}_{\mathrm{eq}}$ error estimators {(converged values computed on the finest asymptotic-regime mesh)}.}
   \label{tab:estimator_compare}
  \begin{tabular}{l c c c c c c c}
    \toprule
    \multirow{2}{*}{Effectivity ratio} & \multicolumn{7}{c}{Polynomial Order} \\
    \cmidrule(lr){2-8}
     & $\mathcal{P}_1$ & $\mathcal{P}_2$ & $\mathcal{P}_3$ & $\mathcal{P}_4$ & $\mathcal{P}_5$ & $\mathcal{P}_6$ & $\mathcal{P}_7$ \\
    \midrule
    $|S^{m}_{\mathcal{P}_p}r|_{H^1(\Omega)}/|u-u_h|_{H^1(\Omega)}$ 
     & 0.8632 & 0.9067 & 0.8279 & 0.7641 & 0.7183 & 0.7391 & 0.7988 \\
    \midrule
    $\langle r, S^a_{\mathcal{P}_p} r \rangle^{1/2}/|u-u_h|_{H^1(\Omega)}$ 
      & 0.8899 & 0.9184 & 0.8953 & 0.8464 & 0.5690 & 0.6166 & 0.7255 \\
    $|S^{a}_{\mathcal{P}_p}r|_{H^1(\Omega)}/|u-u_h|_{H^1(\Omega)}$ 
      & 0.8705 & 0.9754 & 0.9665 & 0.9069 & 0.8610 & 0.8481 & 1.2495\\
    \midrule
    $\mathcal{E}_{\rm{res}}/|u-u_h|_{H^1(\Omega)}$ 
      & 4.9003 & 7.7551 & 12.8526 & 19.5604 & 26.8670 & 34.6710 & 43.0182 \\
    \midrule
    $\mathcal{E}_{\mathrm{eq}}/|u-u_h|_{H^1(\Omega)}$ 
      & 1.0608 & 1.0230 & 1.0250 & 1.0365 & 1.0450 & 1.0520 & 1.0507\\
    \bottomrule
  \end{tabular}
 }
\end{table}

\begin{table}[!th]
  \centering
{\caption{Effectivity ratios of the smoother-type error estimator $|S^{m,\text{iter}}_{\mathcal{P}_p}r|_{H^1(\Omega)}$ based on $1\leq\text{iter}\leq4$ steps of Gauss--Seidel smoothing operations {(converged values computed on the finest asymptotic-regime mesh)}.}
  \label{tab:smoother_iter} 
  \begin{tabular}{c c c c c c c c}
    \toprule
    \multirow{2}{*}{Iter.} & \multicolumn{7}{c}{Polynomial Order} \\
    \cmidrule(lr){2-8}
     & $\mathcal{P}_1$ & $\mathcal{P}_2$ & $\mathcal{P}_3$ & $\mathcal{P}_4$ & $\mathcal{P}_5$ & $\mathcal{P}_6$ & $\mathcal{P}_7$ \\
    \midrule
    1 & 0.8632 & 0.9067 & 0.8279 & 0.7641 & 0.7183 & 0.7391 & 0.7988 \\
    2 & 0.8641 & 0.9423 & 0.9143 & 0.8688 & 0.8304 & 0.7961 & 0.7378 \\
    3 & 0.8643 & 0.9405 & 0.9574 & 0.9185 & 0.8893 & 0.8415 & 0.7833 \\
    4 & 0.8643 & 0.9480 & 0.9759 & 0.9421 & 0.9130 & 0.8627 & 0.8038 \\
    \bottomrule
  \end{tabular}
 }
\end{table}

\begin{remark}
The theoretical $p$-robustness of smoother-type a posteriori error estimates introduced in this work remains an open question because the stable decomposition and contraction in the analysis are not proved to be $p$-robust. Currently, only equilibrated error estimators have been proven to be $p$-robust for specific problems (cf.~\cite{BraessPillweinSchoberl2009,ErnVohralik2020,ChaumontVohralik2023,Li2025arxiv}), owing to their carefully designed patchwise local problems and technical polynomial lifting tools.
\end{remark}

\subsection{A Posteriori Error Estimates in H(curl)}\label{subsect:Hcurl} Smoother-type a posteriori error estimation could be generalized to vector-valued finite elements. In $\mathbb{R}^d$ with $d=2$ or 3, we consider the problem: find $u\in H({\rm curl},\Omega)$ and $u_h\in V_h\subset H({\rm curl},\Omega)$
\begin{subequations}\label{eq:curlcurl}
\begin{align}
a(u,v)=({\rm curl} u,{\rm curl}v)+(u,v)&=(f,v),\quad v\in H({\rm curl},\Omega),\\
\langle A_h({\rm curl})u_h,v\rangle=({\rm curl}u_h,{\rm curl}v)+(u_h,v)&=(f,v),\quad v\in V_h=V_h(\rm curl),\label{eq:curlcurlFEM}
\end{align}
\end{subequations}
where $V_h(\rm curl)$ is a N\'ed\'elec edge element space, and  {$H({\rm curl},\Omega)=\{v\in[L^2(\Omega)]^{d}: {\rm curl} v\in[L^2(\Omega)]^{d(d-1)/2}\}$} is equipped with the norm $\|\bullet\|_{H(\rm curl,\Omega)}=\big(\|\bullet\|_{L^2(\Omega)}^2+\|{\rm curl}\bullet\|_{L^2(\Omega)}^2\big)^\frac{1}{2}$. Here ${\rm curl} v=(\frac{\partial v_3}{\partial x_2}-\frac{\partial v_2}{\partial x_3},\frac{\partial v_1}{\partial x_3}-\frac{\partial v_3}{\partial x_1},\frac{\partial v_2}{\partial x_1}-\frac{\partial v_2}{\partial x_2})$ in $\mathbb{R}^3$ and ${\rm curl}v=\frac{\partial v_2}{\partial x_1}-\frac{\partial v_1}{\partial x_2}$ in $\mathbb{R}^2$.

Smoothers used in Section \ref{subsect:Hgrad} are called pointwise smoothers because they are equivalent to solving one-dimensional local problems on each vertex patch on $\mathcal{T}_{h/2}$. However, pointwise smoothers are inefficient for $H(\rm curl)$ and $H(\rm div)$ problems (cf.~\cite{ArnoldFalkWinther2000,Hiptmair1999SINUM}). Therefore, we consider a block smoother based on the following vertex-oriented subspace decomposition (cf.~\cite{ArnoldFalkWinther2000})
\[
V_{h/2}({\rm curl})=\sum_{k=1}^{N_{h/2}}\big(V_k({\rm curl})=\{v\in V_{h/2}({\rm curl}): {\rm supp}(v)\subset\Omega_{h/2,k}\}\big).
\]
Let
$r=f-A_{h/2}({\rm curl})u_h\in V_{h/2}(\rm curl)^*$.
Let $S_a({\rm curl}): V_{h/2}({\rm curl})^*\rightarrow V_{h/2}({\rm curl})$ be the additive block smoother on $\mathcal{T}_{h/2}$ defined by $S_a({\rm curl})r=\sum_{k=1}^{N_{h/2}}\eta_k$ with each $\eta_k\in V_k({\rm curl})$ solving
\begin{align*}
&({\rm curl}\eta_k,{\rm curl}v)_{\Omega_{h/2,k}}+(\eta_k,v)_{\Omega_{h/2,k}}\\
&=\langle r,v\rangle=(f,v)_{\Omega_{h/2,k}}-({\rm curl} u_h,{\rm curl}v)_{\Omega_{h/2,k}}-(u_h,v)_{\Omega_{h/2,k}}
,\quad v\in V_k({\rm curl}).    
\end{align*}
Let $I_h: V_h({\rm curl})\hookrightarrow V_{h/2}({\rm curl})$ be the inclusion.
It has been shown in \cite{ArnoldFalkWinther2000} that $V_{h/2}({\rm curl})=V_h({\rm curl})+\sum_{k=1}^{N_{h/2}}V_k({\rm curl})$ is a stable subspace decomposition and $$B_a({\rm curl})=I_hA_h({\rm curl})^{-1}I_h^*+S_a({\rm curl})$$ is a uniform preconditioner for $A_{h/2}({\rm curl})$. As a result, the block smoother $S_a=S_a({\rm curl})$ in Theorem \ref{thm:rSr} leads to a uniform a posteriori error estimate 
\begin{equation}\label{eq:additive_Hcurl}
\|u-u_h\|_{H({\rm curl},\Omega)}\eqsim\langle r,S_a({\rm curl})r\rangle^\frac{1}{2}=\Big(\sum_{k=1}^{N_{h/2}}\|\eta_k\|^2_{H({\rm curl},\Omega_{h/2,k})}\Big)^\frac{1}{2}.
\end{equation}
Similarly, the smoother $S_m=S_m({\rm curl})$ in Algorithm \ref{alg:SSC} and Corollary \ref{cor:Srnorm} lead to 
\begin{equation}\label{eq:multiplicative_Hcurl}
\|u-u_h\|_{H({\rm curl},\Omega)}\eqsim\|S_m({\rm curl})r\|_{H({\rm curl},\Omega)}.
\end{equation}
Here $S_a({\rm curl})$ and $S_m({\rm curl})$ are called block smoothers because their implementations require solving nontrivial local problems on each vertex patch $\Omega
_{h/2,k}$.

\begin{figure}[thp]
    \centering
\includegraphics[width=5.5cm]{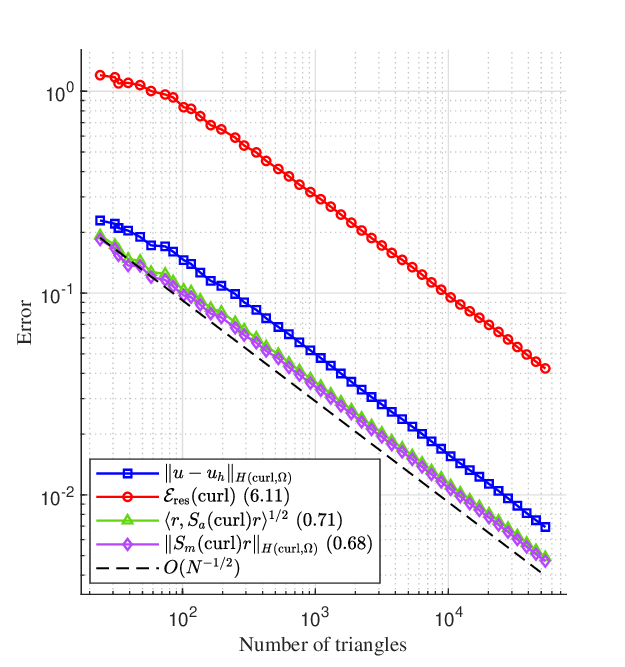} 
    \includegraphics[width=5.5cm]{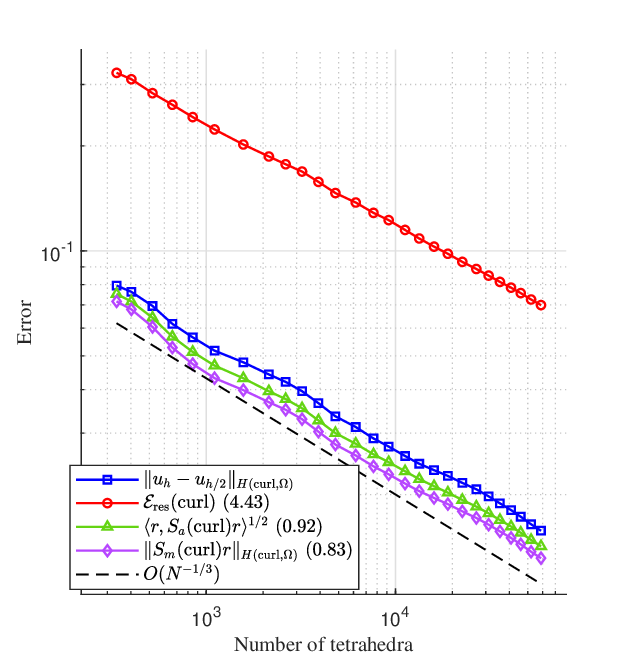}
\caption{Convergence of error estimators for $H({\rm curl})$ problems in 2D (left) and 3D (right).}\label{fig:curlcurl}
\end{figure}

\begin{figure}[th]
    \centering   
    \includegraphics[width=5.5cm]{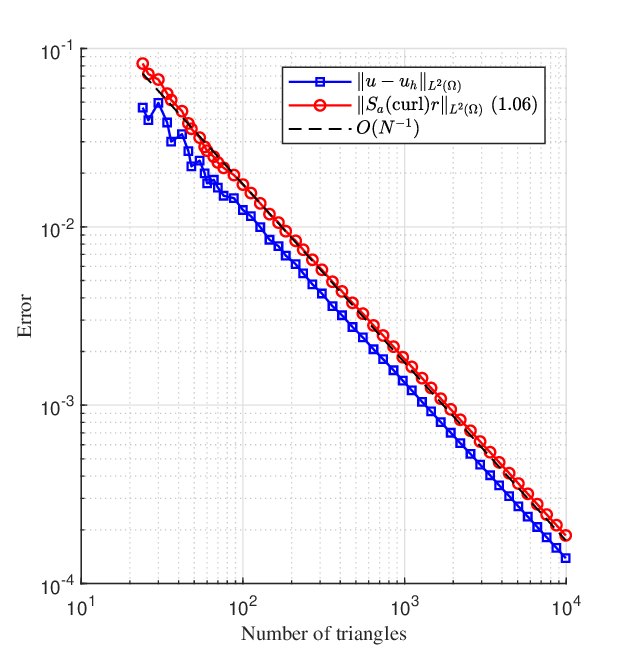}\includegraphics[width=5.5cm]{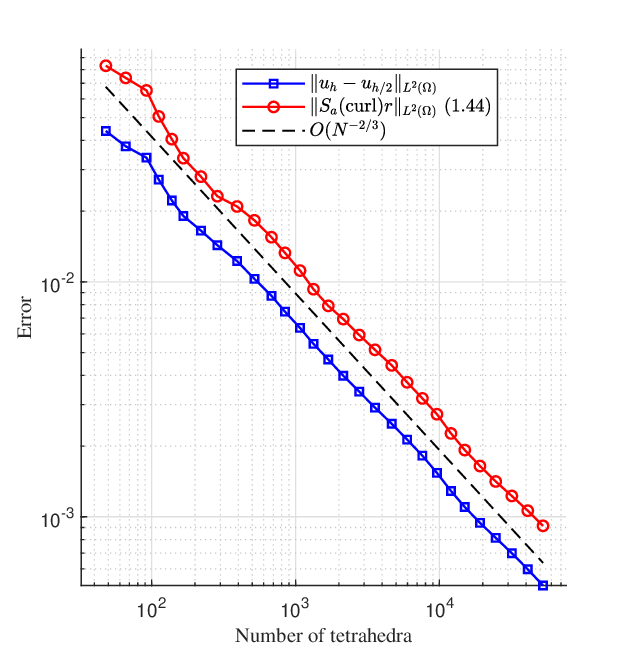}
    \caption{Convergence of a posteriori error estimates for controlling the $L^2$ norm error in 2D (left) and 3D (right).}
    \label{fig:curlcurlL2}
\end{figure}

\begin{figure}[th]
    \centering   
    \includegraphics[width=6cm]{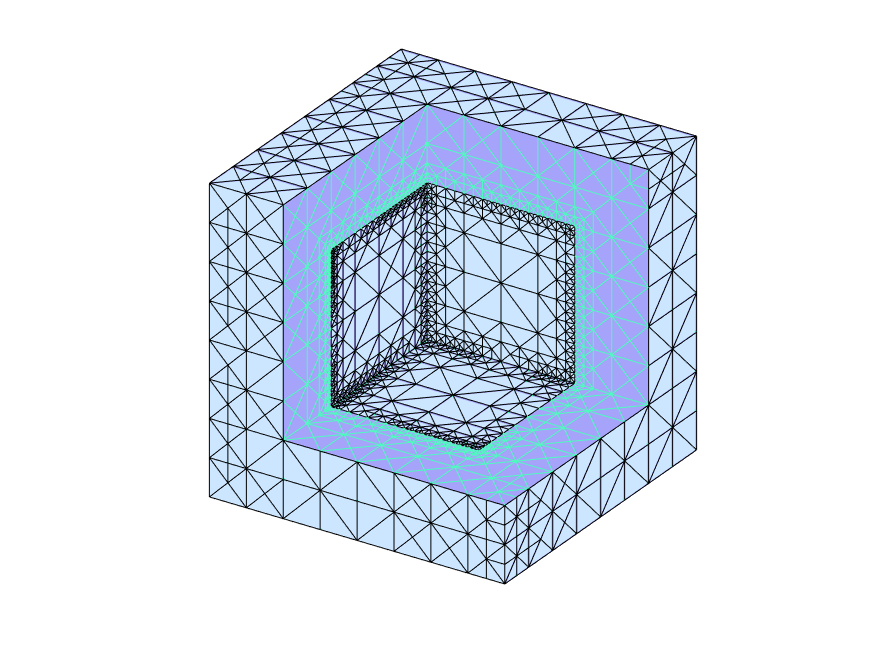}\includegraphics[width=6cm]{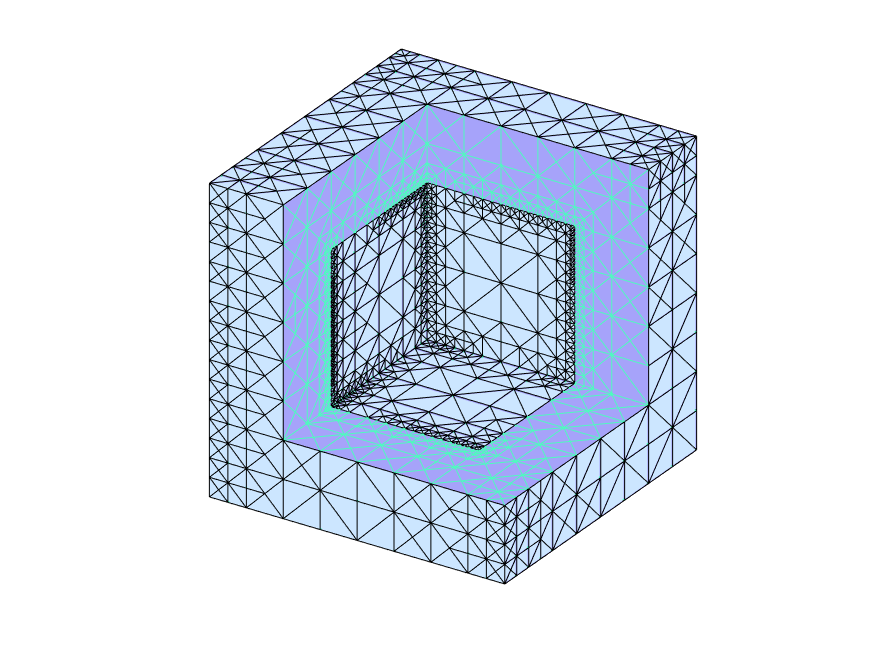}
    \caption{Adaptive meshes (approximately $2.6\times10^4$ tetrahedra) on the 3D cavity region generated by error estimators controlling the $H({\rm curl})$ norm (left) and $L^2$ norm (right)  errors.}
    \label{fig:Hcurl_3dgrids}
\end{figure}

The $H(\rm{curl})$ problem \eqref{eq:curlcurl} is discretized by the lowest order edge finite element on the L-shaped domain $\Omega=(-1,1)^2\backslash([0,1)\times[-1,0))$ as well as on the 3D cavity region $\Omega=(0,1)^3\backslash[0.25,0.75]^3$, see Figure \ref{fig:Hcurl_3dgrids}. We compare the smoother-type a posteriori error estimates \eqref{eq:additive_Hcurl} and \eqref{eq:multiplicative_Hcurl} with the residual-type error estimator (cf.~\cite{BHHW2000,Schoberl2008,LiZikatanov2025mcom})
\begin{align*}
\mathcal{E}_{{\rm res}}({\rm curl})&=\Big(\sum_{T\in\mathcal{T}_h}h_T^2\|{\rm div}(f-u_h)\|_{L^2(T)}^2+h_T^2\|f-{\rm curlcurl}u_h-u_h\|_{L^2(T)}^2\\
&\quad+\sum_{F\text{ is a  face in }\mathcal{T}_h}h_F\|\llbracket(f-u_h)\cdot n\rrbracket\|^2_{L^2(F)}+h_F\|\llbracket n\wedge{\rm curl}u_h\rrbracket\|^2_{L^2(F)}\Big)^\frac{1}{2}.
\end{align*}
Here $n\wedge{\rm curl}u_h|_F=t\cdot{\rm curl}u_h$ in $\mathbb{R}^2$ with $t$ being a unit tangent vector along $F$ and $n\wedge{\rm curl}u_h|_F=n\times{\rm curl}u_h$ in $\mathbb{R}^3$ with $n$ being the unit normal to $F$.
On the L-shaped domain $\Omega$, the exact solution is set to be $u=\nabla(r^{2/3}\sin(2\theta/3))$, where $r^{2/3}\sin(2\theta/3)$ is the singular solution used in Section \ref{subsect:Hgrad}. On the 3D test region, we set $f=(\sqrt{x^2 - 2y^2 + 3z^2 + 3}\exp(-xyz),0,0)$ in \eqref{eq:curlcurl} and use $\|u_h-u_{h/2}\|_{H(\rm curl,\Omega)}$ as the reference FEM error as the exact solution $u$ is unknown. It is shown in Figure \ref{fig:curlcurl} that all estimators lead to optimally convergent FEM solutions, while smoother-type a posteriori error estimates are much closer to the exact error than the residual error estimators. 

We are also interested in a posteriori error control of the $L^2$ norm error $\|u-u_h\|_{L^2(\Omega)}$ of \eqref{eq:curlcurlFEM} based on N\'ed\'elec's edge finite elements of the second kind. In this scenario, $\|u-u_h\|_{L^2(\Omega)}$ converges at a rate one order higher than $\|{\rm curl}(u-u_h)\|_{L^2(\Omega)}$. Currently, we are not aware of any residual-type a posteriori error estimate for $\|u-u_h\|_{L^2(\Omega)}$.  Motivated by the $L^2$ norm error estimators in Section \ref{subsect:Hgrad}, we employ $\|S_a({\rm curl})r\|_{L^2(\Omega)}$ to guide the local mesh refinement of the lowest order N\'ed\'elec's edge element of the second kind. As shown in Figure~\ref{fig:curlcurlL2}, the AFEM driven by $\|S_a({\rm curl})r\|_{L^2(\Omega)}$ recovers optimal convergence rate in the $L^2$ norm. 

\section{Block Diagonal Estimators for Symmetric Systems}\label{sect:saddle} With the help of block diagonal preconditioning (cf.~\cite{MardalWinther2011}), it is not difficult to generalize the smoother-type a posteriori error estimates in Section \ref{sect:discreteV} to saddle point problems. Let $\Sigma$ and $U$ be Hilbert spaces. The model saddle point problem finds $\sigma\in\Sigma$ and $u\in U$ such that  
\begin{subequations}\label{eq:model_saddle}
    \begin{align}
    a(\sigma,\tau)+b(\tau,u)&=g(\tau),\quad \tau\in \Sigma,\\
    b(\sigma,v)&=f(v),\quad v\in U.
\end{align}
\end{subequations}
Here $a: \Sigma\times\Sigma\rightarrow\mathbb{R}$, $b: \Sigma\times U\rightarrow\mathbb{R}$ are bounded bilinear forms, $a$ is symmetric and $f\in U^*$, $g\in \Sigma^*$. Let $\Sigma_h\times U_h\subset\Sigma\times U$ be an inf-sup stable finite element space pair equipped with the inner products $\langle A_{\Sigma_h}\bullet,\bullet\rangle$ and $\langle A_{U_h}\bullet,\bullet\rangle$, respectively, where $A_{\Sigma_h}: \Sigma_h\rightarrow\Sigma_h^*$ and $A_{U_h}: U_h\rightarrow U_h^*$ are SPD. A mixed FEM for \eqref{eq:model_saddle} seeks $\sigma_h\in\Sigma_h$ and $u_h\in U_h$ satisfying 
\begin{subequations}\label{eq:model_MFEM}
    \begin{align}
    a(\sigma_h,\tau)+b(\tau,u_h)&=g(\tau),\quad \tau\in \Sigma_h,\\
    b(\sigma_h,v)&=f(v),\quad v\in U_h.
\end{align}
\end{subequations}
The saturation assumption \eqref{eq:saturation} for \eqref{eq:model_MFEM} reads
\begin{equation}
    \|(\sigma-\sigma_{h/2},u-u_{h/2})\|_{\Sigma\times U}\leq\gamma_{h/2}\|(\sigma-\sigma_h,u-u_h)\|_{\Sigma\times U}
\end{equation}
for some $\gamma_{h/2}<1$.
We consider the augmented bilinear form $\mathcal{B}(\sigma,u;\tau,v):=a(\sigma,\tau)+b(\tau,u)+b(\sigma,v)$ and fix the boundedness and inf-sup constants $\alpha_h>0$, $\beta_h>0$  satisfying
\begin{align*}
    \mathcal{B}(\tau_1,v_1;\tau_2,v_2)&\leq\alpha_h\|(\tau_1,v_1)\|_{\Sigma\times U}\|(\tau_2,v_2)\|_{\Sigma\times U},\\
    \sup_{\|(\tau_2,v_2)\|_{\Sigma\times U}=1}\mathcal{B}(\tau_1,v_1;\tau_2,v_2)&\geq \beta_h\|(\tau_1,v_1)\|_{\Sigma\times U}
\end{align*}
for all $(\tau_1,v_1), (\tau_2,v_2)\in\Sigma\times U$. Let the residuals $r_\sigma\in\Sigma_{h/2}^*$, $r_u\in U_{h/2}^*$ be given by \begin{align*}
\langle r_\sigma,\tau\rangle&:=g(\tau)-a(\sigma_h,\tau)-b(\tau,u_h),\quad\tau\in\Sigma_{h/2},\\      
\langle r_u,v\rangle&:=f(v)-b(\sigma_h,v),\quad v\in U_{h/2}.
\end{align*}
In the next theorem, we propose a framework for constructing smoother-type a posteriori error estimates of \eqref{eq:model_MFEM}.
\begin{theorem}\label{thm:rSr_saddle}
Assume that $\Sigma_h\subset\Sigma_{h/2}$ and $U_h\subset U_{h/2}$ with $I_h: \Sigma_h\rightarrow\Sigma_{h/2}$ and $J_h: U_h\rightarrow U_{h/2}$ being inclusions. Let $B_{\Sigma_{h/2}}=I_h A_{\Sigma_{h/2}}^{-1}I_h^*+S_{\Sigma_{h/2}}: \Sigma_{h/2}^*\rightarrow \Sigma_{h/2}$ and $B_{U_{h/2}}=J_h A_{U_{h/2}} ^{-1}J_h^*+S_{U_{h/2}}: U_{h/2}^*\rightarrow U_{h/2}$.
Assume that $B_{\Sigma_{h/2}}$, $B_{U_{h/2}}$ are preconditioners for $A_{\Sigma_{h/2}}$, $A_{U_{h/2}}$, respectively, i.e.,
\begin{align*}
&c_2\langle r_1,B_{\Sigma_{h/2}} r_1\rangle\leq\langle r_1,A_{\Sigma_{h/2}}^{-1}r_1\rangle\leq c_3\langle r_1,B_{\Sigma_{h/2}} r_1\rangle,\quad\forall r_1\in \Sigma_{h/2}^*,\\
&c_4\langle r_2,B_{U_{h/2}}r_2\rangle\leq\langle r_2,A_{U_{h/2}}^{-1}r_2\rangle\leq c_5\langle r_2,B_{U_{h/2}}r_2\rangle,\quad\forall r_2\in U_{h/2}^*,
\end{align*}
where $c_2, ..., c_5>0$ are uniform constants. Then for \eqref{eq:model_MFEM} it holds that 
\begin{equation*}
\frac{\sqrt{ \min(c_2,c_4) }}{\alpha_{h/2}(1+\gamma_{h/2})}\mathcal{E}_h(\sigma_h,u_h)\leq\|(\sigma-\sigma_h,u-u_h)\|_{\Sigma\times U}\leq\frac{\sqrt{ \max(c_3,c_5) }}{\beta_{h/2}(1-\gamma_{h/2})}\mathcal{E}_h(\sigma_h,u_h).
\end{equation*}   
where the error estimator is
\begin{align*}
\mathcal{E}_h(\sigma_h,u_h)&=\sqrt{\langle r_\sigma,B_{\Sigma_{h/2}}r_\sigma\rangle+\langle r_u,B_{U_{h/2}}r_u\rangle}\\
&=\sqrt{\langle r_\sigma,S_{\Sigma_{h/2}} r_\sigma\rangle+\langle r_u,S_{U_{h/2}}r_u\rangle}.
\end{align*}   
\end{theorem}
\begin{proof}
Using $\Sigma_h\subset\Sigma_{h/2}$, $U_h\subset U_{h/2}$, we have  $I_h^*r_\sigma=0$, $J_h^*r_u=0$ and thus $\langle r_\sigma,B_{\Sigma_{h/2}}r_\sigma\rangle=\langle r_\sigma,S_{\Sigma_{h/2}}r_\sigma\rangle$ and $\langle r_u,B_{U_{h/2}}r_u\rangle=\langle r_u,S_{U_{h/2}}r_u\rangle$. By definition of the inf-sup condition and spectral equivalence assumption,  we have
\begin{align*}
&\|(\sigma_{h/2}-\sigma_h,u_{h/2}-u_h)\|_{\Sigma\times U}\leq{\beta_{h/2}^{-1}}\|(r_\sigma,r_u)\|_{\Sigma^*\times U^*}\\
&=\beta_{h/2}^{-1}\sqrt{\langle r_\sigma,A_{\Sigma_{h/2}}^{-1}r_\sigma\rangle+\langle r_u,A^{-1}_{U_{h/2}}r_u\rangle}\\
&\leq{\beta_{h/2}^{-1}}\sqrt{c_3\langle r_\sigma,B_{\Sigma_{h/2}} r_\sigma\rangle+c_5\langle r_u,B_{U_{h/2}}r_u\rangle},
\end{align*}  
which implies that
\begin{equation}\label{eq:MFEM_uhuh2}
\|(\sigma_{h/2}-\sigma_h,u_{h/2}-u_h)\|_{\Sigma\times U}\leq \frac{\sqrt{\max(c_3,c_5) }}{\beta_{h/2}}\sqrt{\langle r_\sigma,S_{\Sigma_{h/2}} r_\sigma\rangle+\langle r_u,S_{U_{h/2}}r_u\rangle}.
\end{equation} 
Combining \eqref{eq:MFEM_uhuh2} with the saturation assumption, we obtain the upper bound
\begin{equation*}
\|(\sigma-\sigma_h,u-u_h)\|_{\Sigma\times U}\leq\frac{\sqrt{\max(c_3,c_5) }}{\beta_{h/2}(1-\gamma_{h/2})}\sqrt{\langle r_\sigma,S_{\Sigma_{h/2}} r_\sigma\rangle+\langle r_u,S_{U_{h/2}}r_u\rangle}.
\end{equation*}
The lower bound is proved in a similar way.
\end{proof}

\subsection{Poisson in Mixed Form} 
Let  $\Sigma_h=V_h({\rm div})$ be the lowest order Raviart-Thomas finite element space equipped with the norm $\|\bullet\|_{H(\rm div,\Omega)}=\big(\|\bullet\|_{L^2(\Omega)}^2+\|{\rm div}\bullet\|_{L^2(\Omega)}^2\big)^\frac{1}{2}$ and inner product $\langle A_h({\rm div})\bullet,\bullet\rangle$. Let $U_h$ be the space of piecewise constants on $\mathcal{T}_h$ equipped with the $L^2(\Omega)$ norm. 
{For the boundary value problem $-\Delta u=f$ in $\Omega$ with $u|_{\partial\Omega}=g$,} the mixed FEM seeks $(\sigma_h,u_h)\in V_h({\rm div})\times U_h$ such that 
\begin{equation}\label{eq:Poisson_MFEM}
    \begin{aligned}
       (\sigma_h,\tau_h)+({\rm div}\tau_h,u_h)&={( g, \tau_h\cdot n)_{\partial\Omega}},\quad\tau_h\in V_h({\rm div}),\\
        ({\rm div}\sigma_h,v_h)&=-(f,v_h),\quad v_h\in U_h.
    \end{aligned}
\end{equation}
Here $\sigma_h$ is an approximation to the gradient field $\sigma=\nabla u$. Due to the discontinuous nature of $U_h$, 
the inner product operator $A_{U_{h/2}}$ of $U_{h/2}$ corresponds to a diagonal matrix, which is easily invertible. Similarly to $S_a(\rm curl)$ in Section \ref{subsect:Hcurl}, the smoother for $A_{\Sigma_{h/2}}$ is of the block-type  based on the vertex-oriented subspace  
\[
V_{h/2}({\rm div})=\sum_{k=1}^{N_{h/2}}\big(V_k({\rm div})=\{v\in V_{h/2}({\rm div}): {\rm supp}(v)\subset\Omega_{h/2,k}\}\big).
\]
Let $S_a({\rm div}): V_{h/2}({\rm div})^*\rightarrow V_{h/2}({\rm div})$ be the additive block smoother on $\mathcal{T}_{h/2}$, i.e., $S_a({\rm div})r_\sigma=\sum_{k=1}^{N_{h/2}}\eta_k$ with  each $\eta_k\in V_k({\rm div})$ solving
\begin{align*}
&({\rm div}\eta_k,{\rm div}\tau)_{\Omega_{h/2,k}}+(\eta_k,\tau)_{\Omega_{h/2,k}}=\langle r_\sigma,\tau\rangle\\
&={(g, \tau\cdot n)_{\partial\Omega\cap\partial\Omega_{h/2,k}}}-(\sigma_h,\tau)_{\Omega_{h/2,k}}-({\rm div}\tau,u_h)_{\Omega_{h/2,k}}
,\quad \tau\in V_k({\rm div}).    
\end{align*}
In \cite{ArnoldFalkWinther2000}, it has been shown that $V_{h/2}({\rm div})=V_h({\rm div})+\sum_{k=1}^{N_{h/2}}V_k({\rm div})$ is a stable subspace decomposition and $$B_a({\rm div})=I_hA_h({\rm div})^{-1}I_h^*+S_a({\rm div})$$ is a uniform preconditioner for the discrete $H({\rm div})$ operator $A_{h/2}({\rm div})$. Therefore, for the mixed FEM \eqref{eq:Poisson_MFEM}, 
we have verified the assumptions in Theorem \ref{thm:rSr_saddle} and obtain
\begin{equation}\label{eq:MFEM_smoother}
\begin{aligned}
\|\sigma-\sigma_h\|^2_{H({\rm div},\Omega)}+\|u-u_h\|^2_{L^2(\Omega)}&\eqsim\langle r_\sigma,S_a({\rm div}) r_\sigma\rangle+\langle r_u,A_{U_{h/2}}^{-1}r_u\rangle\\
&=\sum_{k=1}^{N_{h/2}}\|\eta_k\|_{H({\rm div},\Omega_{h/2,k})}^2+\langle r_u,A_{U_{h/2}}^{-1}r_u\rangle.
\end{aligned}
\end{equation}

\begin{figure}[th]
\centering
\includegraphics[width=6cm]{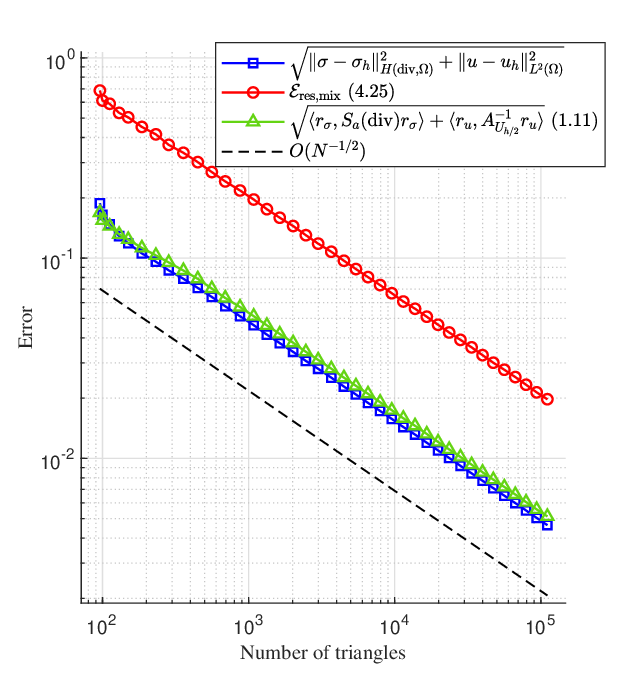} 
\caption{Convergence of a posteriori error estimates for the mixed FEM \eqref{eq:Poisson_MFEM}.}
\label{fig:Poisson_MFEM}
\end{figure}

For the mixed FEM \eqref{eq:Poisson_MFEM}, we compare the smoother-type a posteriori error estimate \eqref{eq:MFEM_smoother} to the residual error estimator (cf.~\cite{BV1996,Carstensen1997,Li2021MCOM,Li2021M2AN})
\begin{align*}
&\mathcal{E}_{{\rm res},{\rm mix}}=\Big(\sum_{T\in\mathcal{T}_h}h_T^2\|\sigma_h-\nabla u_h\|_{L^2(T)}^2+h_T^2\|{\rm curl}\sigma_h\|_{L^2(T)}^2+\|f+{\rm div}\sigma_h\|_{L^2(T)}^2\\
&+\sum_{\text{edge }E\subset\mathring{\Omega}}h_E\|\llbracket\sigma_h\cdot t\rrbracket\|^2_{L^2(E)}+h_E\|\llbracket u_h\rrbracket\|^2_{L^2(E)}{+\sum_{\text{edge }E\subset\partial\Omega}h_E\|\sigma_h\cdot t-\partial g/\partial t\|^2_{L^2(E)}+h_E\|u_h-g\|^2_{L^2(E)}\Big)^\frac{1}{2}}.
\end{align*}
{In the experiment, we use the same exact singular solution $u$ in Section \ref{subsect:NE_Hgrad}.}
Figure \ref{fig:Poisson_MFEM} demonstrates that the smoother-type estimator \eqref{eq:MFEM_smoother} is more accurate than $\mathcal{E}_{\rm res,mix}$ and yields an optimally convergent AFEM.

\subsection{Biot's Consolidation Model}\label{subsect:Biot}
Complex multi-physics problems are often modeled by saddle-point systems involving multiple physical parameters. Motivated by existing parameter-robust solvers in the literature, it is relatively easy to derive parameter-robust smoother-type a posteriori error estimates for multi-physics systems. For instance, we consider the steady-state Biot's consolidation model
\begin{subequations}\label{eq:Biot}
    \begin{align}
    -{\rm div}(\varepsilon(u)+\lambda({\rm div}u) I-\alpha p_F I)&=f\quad\text{ in }\Omega,\\
    -\alpha^2\lambda^{-1}p_F-\alpha{\rm div}u+{\rm div}(\kappa\nabla p_F)&=g\quad\text{ in }\Omega,
\end{align}
\end{subequations}
under boundary condition $u|_{\partial\Omega}=0$, $p_F|_{\partial\Omega}=0$, 
where $\varepsilon(u)=(\nabla u+(\nabla u)^T)/2$ is the symmetric gradient of $u$, and $\lambda>0$, $\alpha>0$, $\kappa>0$ are rescaled Lam\'e, Biot--Willis and permeability constants, respectively. The system \eqref{eq:Biot} models the
interactions between the deformation and fluid flow in a
fluid-saturated elastic porous medium, where $u$ is the displacement of the elastic medium and $p_F$ the fluid pressure. 

For simplicity, we derive smoother-type a posteriori error estimates only for the total pressure formulation \cite{LeeMardalWinther2017,BoonKuchtaMardalRuizBaier2021,Carcamo2024} among many classical discretizations of Biot's equation, see \cite{PhillipsWheeler2007,HongKraus2018,RodrigoHu2018,PGHARZ2025} and references therein. Let $V_h\subset [H_0^1(\Omega)]^d$, $Q_{T,h}\subset L^2(\Omega)$, $Q_{F,h}\subset H_0^1(\Omega)$ be finite element spaces and $V_h\times Q_{T,h}$ be a Stokes stable finite element pair.  For \eqref{eq:Biot}, the mixed FEM utilizing total pressure seeks $u_h\in V_h$, $p_{T,h}\in Q_{T,h}$, $p_{F,h}\in Q_{F,h}$ such that
\begin{subequations}\label{eq:total_pressure}
    \begin{align}
    (\varepsilon(u_h),\varepsilon (v_h))-({\rm div}v_h,p_{T,h})&=(f,v_h),\label{eq:total_pressure1}\\
    -({\rm div}u_h,q_{T,h})-\lambda^{-1}(p_{T,h},q_{T,h})+\alpha\lambda^{-1}(p_{F,h},q_{T,h})&=0,\label{eq:total_pressure2}\\
   \alpha\lambda^{-1}(p_{T,h},q_{F,h})-2\alpha^2\lambda^{-1}(p_{F,h},q_{F,h})-\kappa(\nabla p_{F,h},\nabla q_{F,h})&=(g,q_{F,h})\label{eq:total_pressure3}
\end{align}
\end{subequations}
for all $v_h\in V_h$, $q_{T,h}\in Q_{T,h}$, $q_{F,h}\in Q_{F,h}$. 
Here $p_{T,h}$ is the finite element approximation to the artificial total pressure $p_T:=\alpha p_F-\lambda{\rm div}u$. 
Although a posteriori error estimates for Biot's model have been developed in, e.g.,
\cite{ErnMeunier2009,AhmedRaduNordbotten2019,KhanSilvester2021,WheelerGiraultLi2022,LiZikatanov2022IMA,FumagalliParoliniVerani2025}, a posteriori error analysis for the total pressure formulation \eqref{eq:total_pressure} seems missing in the literature. 

Let $[H_0^1(\Omega)]^d$, $L_0^2(\Omega)$ and $H_0^1(\Omega)$ be equipped with the norms
\begin{align*}
\|v\|_V&=\langle A_Vv,v\rangle^\frac{1}{2}:=\|\varepsilon(v)\|_{L^2(\Omega)},\\\|q_T\|_{Q_T}&=\langle A_{Q_T}q_T,q_T\rangle^\frac{1}{2}:=\sqrt{\lambda^{-1}(q_T,q_T)+(\bar{q}_T,\bar{q}_T)},\\
\|q_F\|_{Q_F}&=\langle A_{Q_F}q_F,q_F\rangle^\frac{1}{2}:=\sqrt{\alpha^2\lambda^{-1}(q_F,q_F)+\kappa(\nabla q_F,\nabla q_F)},
\end{align*}
where $\bar{q}_T:=(q_T,1)/|\Omega|.$ These operators have obvious analogues $A_{V_{h/2}}$, $A_{Q_{T,h/2}}$, $A_{Q_{F,h/2}}$  on the auxiliary refined mesh $\mathcal{T}_{h/2}$.
It has been proved in \cite{LeeMardalWinther2017} that the error 
    \[
    \!|\!|\!|(u,p_T,p_F)-(u_h,p_{T,h},p_{F,h})\,\!|\!|\!|:=\sqrt{\|u-u_h\|_V^2+\|p_T-p_{T,h}\|_{Q_T}^2+\|p_F-p_{F,h}\|_{Q_F}^2}
    \]
of \eqref{eq:total_pressure} is uniformly convergent for all parameters $\alpha, \kappa\leq1$, $\lambda\geq1$. Moreover, parameter-robust inf-sup stability and a block diagonal preconditioner $$\begin{pmatrix}A_{V_{h/2}}^{-1}&O&O\\O&A^{-1}_{Q_{T,h/2}}&O\\O&O&A^{-1}_{Q_{F,h/2}}\\\end{pmatrix}$$ for \eqref{eq:total_pressure} have been developed in \cite{LeeMardalWinther2017}. 

Let $S_{\widetilde{V}}: \widetilde{V}^*\rightarrow \widetilde{V}$ be the pointwise Jacobi smoother, namely, the operator corresponding to the inverse diagonal of the matrix for $A_{\widetilde{V}}$ with $\widetilde{V}=V_{h/2}$ or $Q_{F,h/2}$. Additive two-level preconditioners using the smoother $S_{\widetilde{V}}$ suffice to uniformly precondition the second order elliptic operator $A_{\widetilde{V}}$.
Although $A_{Q_{T,h/2}}$ corresponds to a mass matrix, 
the global average in $A_{Q_{T,h/2}}$ makes the matrix for $A_{Q_{T,h/2}}$ dense.
In Section 5 of \cite{LeeMardalWinther2017},  $A_{Q_{T,{h/2}}}$ is preconditioned by $S_{Q_{T,h/2}}: Q_{T,h/2}^*\rightarrow Q_{T,h/2}$ corresponding to a matrix of the form $(\mathbb{I}+\mathbb{V})\mathbb{D}(\mathbb{I}+\mathbb{V}^\top)$, where $\mathbb{I}$ is the identity, $\mathbb{D}$ is diagonal and $\mathbb{V}$ is of rank one.

Following \cite{LeeMardalWinther2017} and the results in Theorem \ref{thm:rSr_saddle}, we directly obtain the smoother-type error estimator 
\begin{equation}\label{eq:Biot_smoother}
\begin{aligned}
&\!|\!|\!|(u,p_T,p_F)-(u_h,p_{T,h},p_{F,h})\,\!|\!|\!|\eqsim\mathcal{E}_{\rm smoother}\\
&:=\sqrt{\langle r_V,S_{V_{h/2}} r_V\rangle+\langle r_T,S_{Q_{T,h/2}}r_T\rangle+\langle r_F,S_{Q_{F,h/2}}r_F\rangle}
\end{aligned}
\end{equation}  
for \eqref{eq:total_pressure}, where $r_V\in V_{h/2}^*$, $r_T\in Q_{T,h/2}^*$, $r_F\in Q_{F,h/2}^*$ are the finite element residuals corresponding to \eqref{eq:total_pressure1}, \eqref{eq:total_pressure2}, \eqref{eq:total_pressure3} on $\mathcal{T}_{h/2}$, respectively.

\begin{figure}[thp]
    \centering
    \includegraphics[width=4cm]{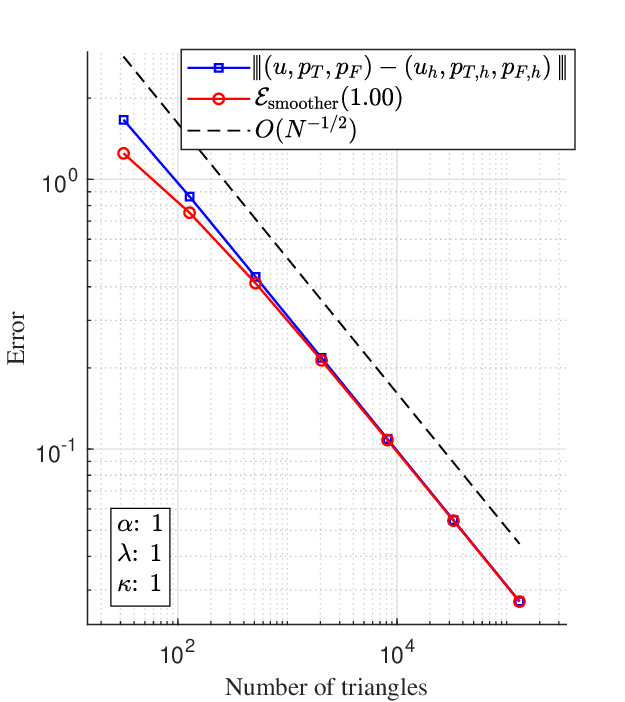}
     \includegraphics[width=4cm]{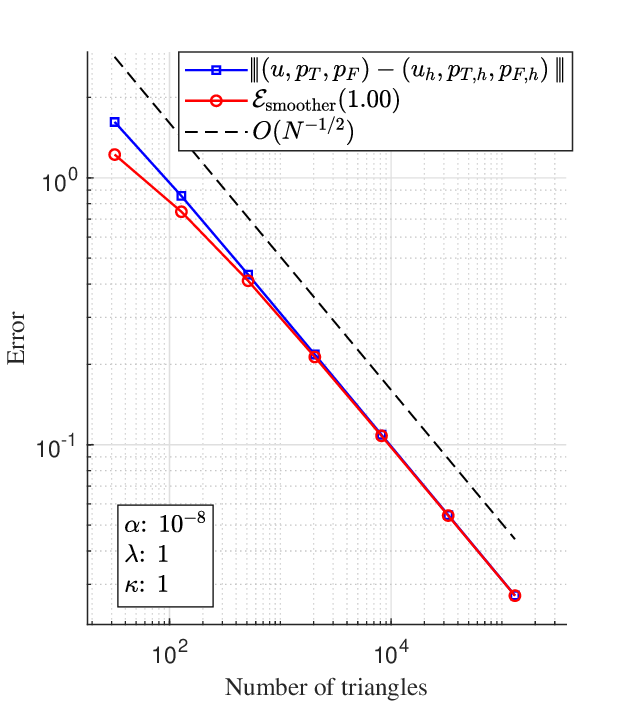}
     \includegraphics[width=4cm]{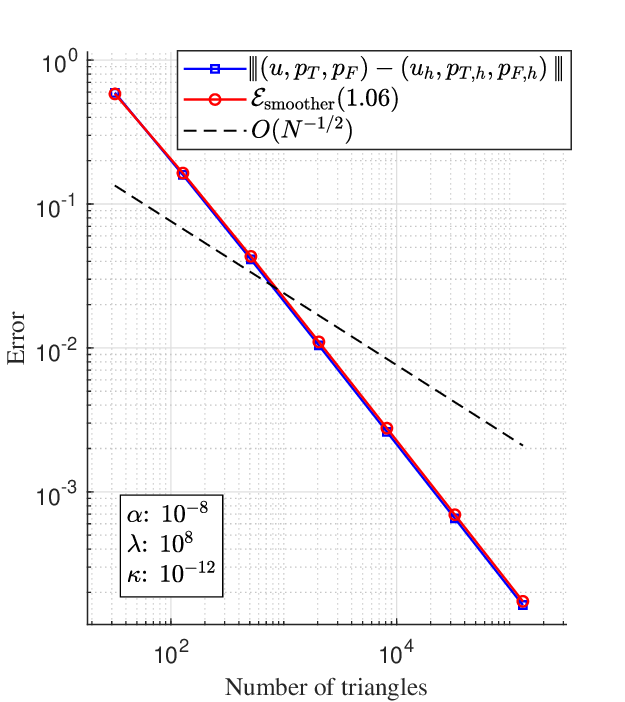}
\caption{Uniform convergence of error estimators for \eqref{eq:Biot} using a constructed solution with effectivity ratio estimator/error shown in the parenthesis.}\label{fig:Biot}
\end{figure}

\begin{figure}[th]
    \centering
\includegraphics[width=4cm]{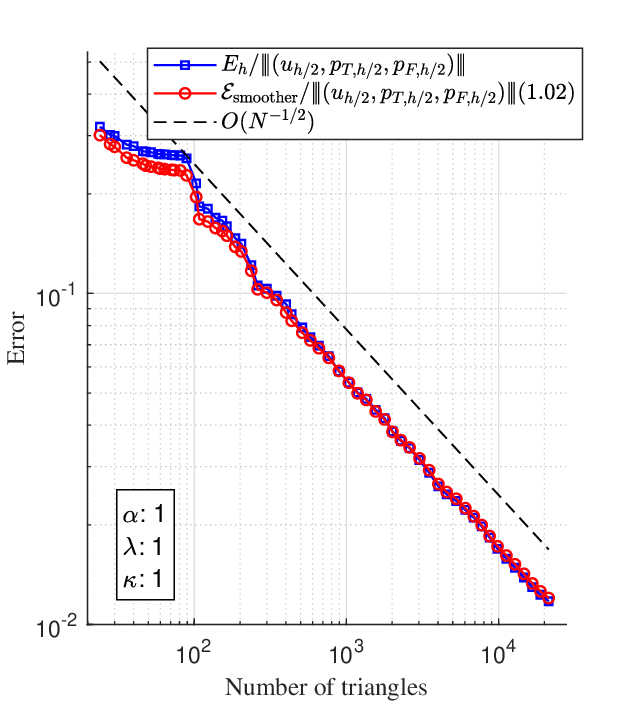}
     \includegraphics[width=4cm]{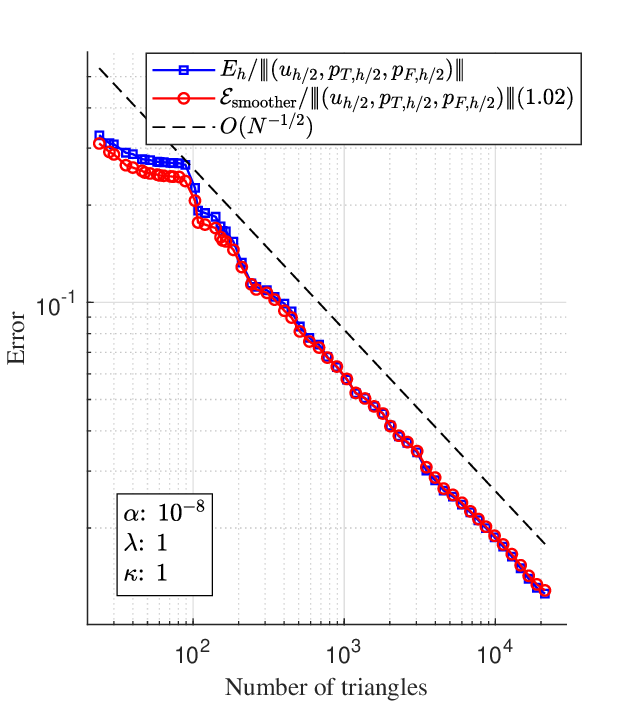}
     \includegraphics[width=4cm]{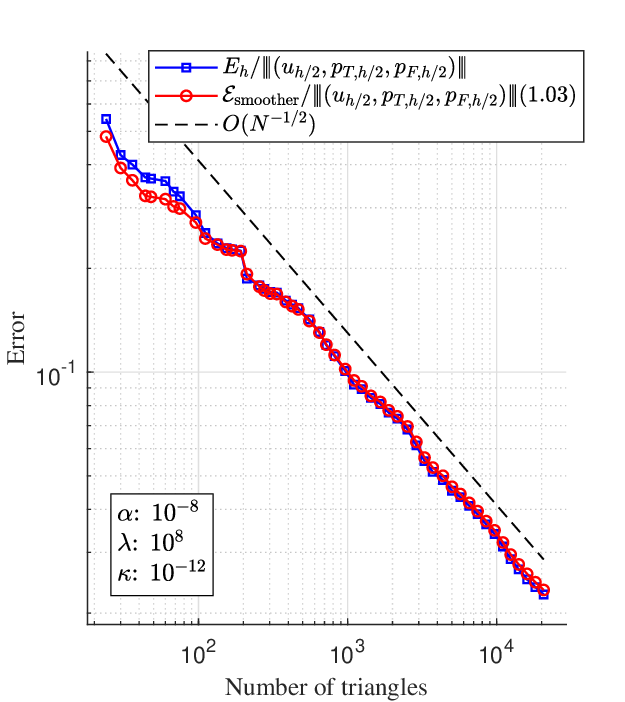}
\caption{Convergence of the AFEM driven by a smoother-type error estimator for \eqref{eq:Biot} on an L-shaped domain with $f=(1,1)$, $g=1$ and ratio estimator/error shown in the parenthesis.}\label{fig:Biot_AFEM}
\end{figure}

In the numerical experiment, we set $V_h$, $Q_{T,h}$, $Q_{F,h}$ as  $\mathcal{P}_2$, $\mathcal{P}_1$, $\mathcal{P}_1$ continuous finite element spaces, respectively. On the unit square $\Omega = (0,1)^2$, the Dirichlet boundary condition and the source terms  $f$ and $g$ are manufactured from the exact solution:
\begin{equation*}
    u = \big( \sin(\pi x) \cos(\pi y), -\cos(\pi x) \sin(\pi y) \big), \quad p_F = \sin(\pi x) \sin(\pi y).
\end{equation*}
We test three sets of parameters $(\alpha, \lambda, \kappa)=(1, 1, 1)$, $(10^{-8}, 1, 1)$, $(10^{-8}, 10^8, 10^{-12})$ under uniform mesh refinement. It is observed from  Figure~\ref{fig:Biot} that $\mathcal{E}_{\rm smoother}$ is quite close to $\!|\!|\!|(u,p_T,p_F)-(u_h,p_{T,h},p_{F,h})\,\!|\!|\!|$ and the effectivity ratio $\mathcal{E}_{\rm smoother}/\allowbreak \,\!|\!|\!|(u,p_T,p_F)-\allowbreak (u_h,p_{T,h},p_{F,h})\,\!|\!|\!|$ is robust even for extreme parameter values. 

Then we perform adaptive simulations on the L-shaped domain $\Omega = (-1,1)^2 \setminus \big( [0,1) \times [-1,0) \big)$, with source terms $f = (1,1)$ and $g=1$. Since the exact solution is unknown, we use $$E_h=\!|\!|\!|(u_h,p_{T,h},p_{F,h})-(u_{h/2},p_{T,h/2},p_{F,h/2})\,\!|\!|\!|$$ as the reference error. The corresponding numerical results are presented in Figure~\ref{fig:Biot_AFEM}.
The smoother-type error estimators are able to accurately control the FEM error under locally refined meshes for singular solutions.

\section{Error Estimators for Non-symmetric Problems}\label{sect:nonsymmetric}
A rigorous study of smoother-type a posteriori error estimates for non-symmetric problems is beyond the scope of the current paper.  Nevertheless, error estimators such as $\|S^a_{h/2}r\|$ are still applicable to arbitrary linear PDEs.

\subsection{Convection-diffusion equation}\label{subsect:convection_diffusion}
Let $\alpha>0$ be the diffusion coefficient and  $\beta: \Omega\rightarrow\mathbb{R}^d$ be a convective field. The convection-diffusion equation 
\begin{equation}\label{eq:convection_diffusion}
\begin{aligned}
-\nabla\cdot(\alpha\nabla u+\beta u)&=f\quad\text{ in } \Omega,\\
   u&=0\quad\text{ on }\partial\Omega
\end{aligned}
\end{equation}
is notoriously difficult to solve due to the sharp boundary layer near $\partial\Omega$ when $\alpha\ll\|\beta\|_{L^\infty(\Omega)}$. Let $V_h\subset H_0^1(\Omega)$ be the continuous and piecewise $\mathcal{P}_p$ finite element subspace. The FEM for \eqref{eq:convection_diffusion} seeks $u_h \in V_h$ such that
\begin{equation*}
    a(u_h, v_h)=\langle A_hu_h,v_h\rangle:=(\alpha \nabla u_h, \nabla v_h) + (\beta u_h,\nabla v_h) = (f, v_h), \quad v_h \in V_h.
\end{equation*}
For convection-dominated problems, it is natural to resolve the boundary layer using adaptive mesh refinements. 

Let $S^a_{\mathcal{P}_p} $ and $S^m_{\mathcal{P}_p}$ be the pointwise Jacobi and Gauss--Seidel smoothers for $A_{h/2}$, respectively. 
Following the ideas in Section \ref{subsect:Hgrad}, we set $r=f-A_{h/2}u_h\in V_{h/2}^*$ and use $|S^a_{\mathcal{P}_p}r|_{H^1(\Omega)}$ and $|S^m_{\mathcal{P}_p}r|_{H^1(\Omega)}$ as a posteriori error estimates of  AFEMs for \eqref{eq:convection_diffusion}. In the numerical example, $\Omega = (0,1)^2$,  $\beta = (1,2)$ and $f=1$. The  diffusion coefficient is a piecewise constant:
\begin{equation*}
    \alpha(x)= 
    \begin{cases}
        1,       & x_1 \leq 0.5, \\
        10^{-3}, & x_1 > 0.5.
    \end{cases}
\end{equation*}
The polynomial degree $p$ ranges from 1 to 7. The initial mesh size $h_0$ is set as $h_0 = 2^{-7}$ for $p \leq 2$ and $h_0 = 2^{-5}$ for $p\geq3$, as low order AFEMs starting from a very coarse initial grid fail to converge.

The experimental results in Figure \ref{fig:Convection_diffusion_Pk} demonstrate that the smoother-type estimators remain effective for the highly non-symmetric problem \eqref{eq:convection_diffusion}. As shown in Figure \ref{fig:Convection_diffusion_GS_P1P6_mesh}, the corresponding AFEM is able to capture both the interface and boundary singularity. It is also observed from Figure \ref{fig:Convection_diffusion_Pk} that the effectivity ratios $|S^a_{\mathcal{P}_p}r|_{H^1(\Omega)}/|u-u_h|_{H^1(\Omega)}$ and $|S^m_{\mathcal{P}_p}r|_{H^1(\Omega)}/|u-u_h|_{H^1(\Omega)}$ remain mild as the polynomial degree $p$ grows.

\begin{figure}[thp]
\centering
\includegraphics[width=6cm]{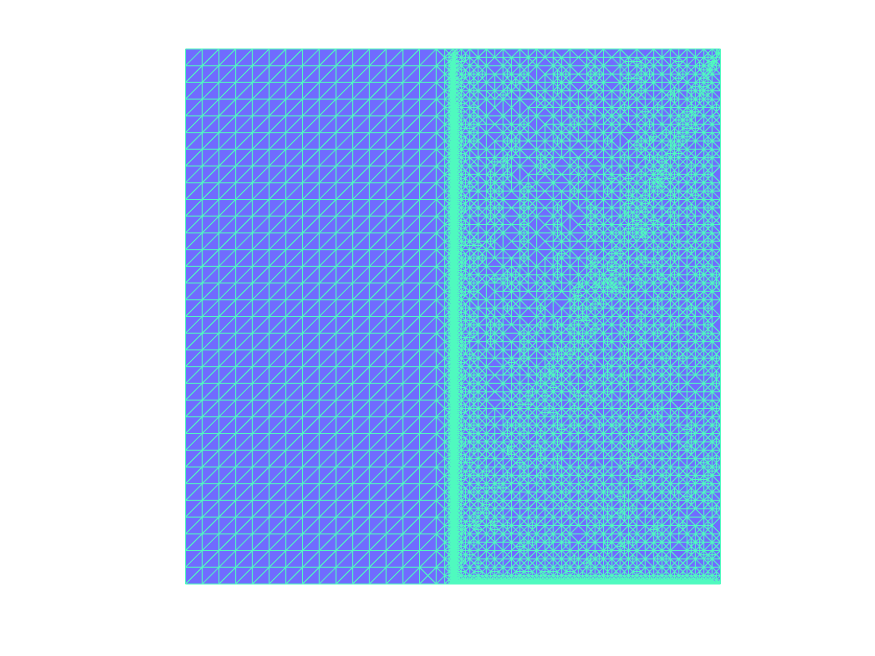}
\includegraphics[width=6cm]{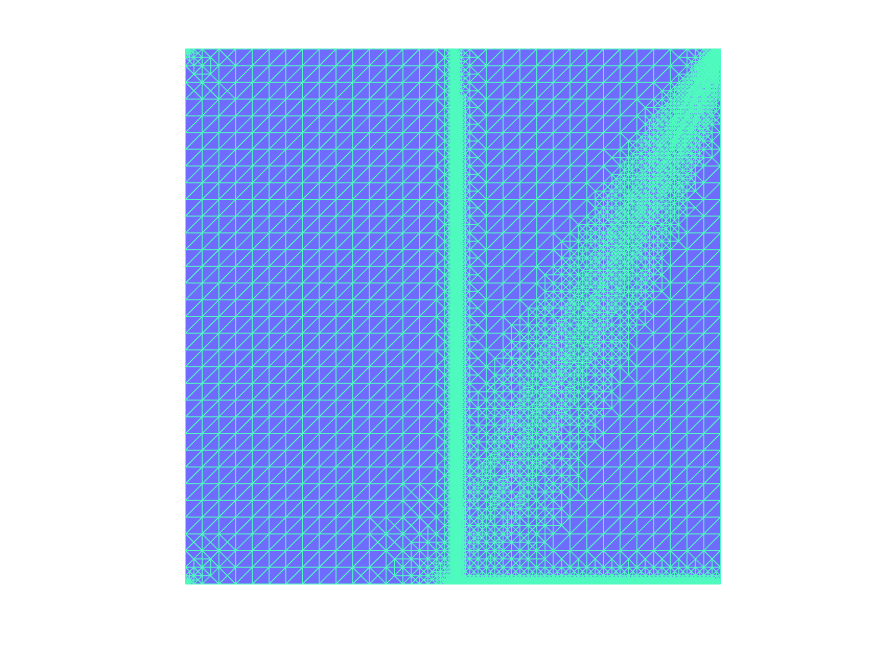} 
    \caption{Adaptive meshes with  approximately $2.2\times10^4$ triangles produced by $\mathcal{P}_1$- (left) and $\mathcal{P}_5$-AFEM (right) driven by the Gauss--Seidel smoother-type error estimator $|S^m_{\mathcal{P}_p}r|_{H^1(\Omega)}$.}
    \label{fig:Convection_diffusion_GS_P1P6_mesh}
\end{figure}

\begin{figure}[h]
        \centering
        \includegraphics[width=5.5cm]{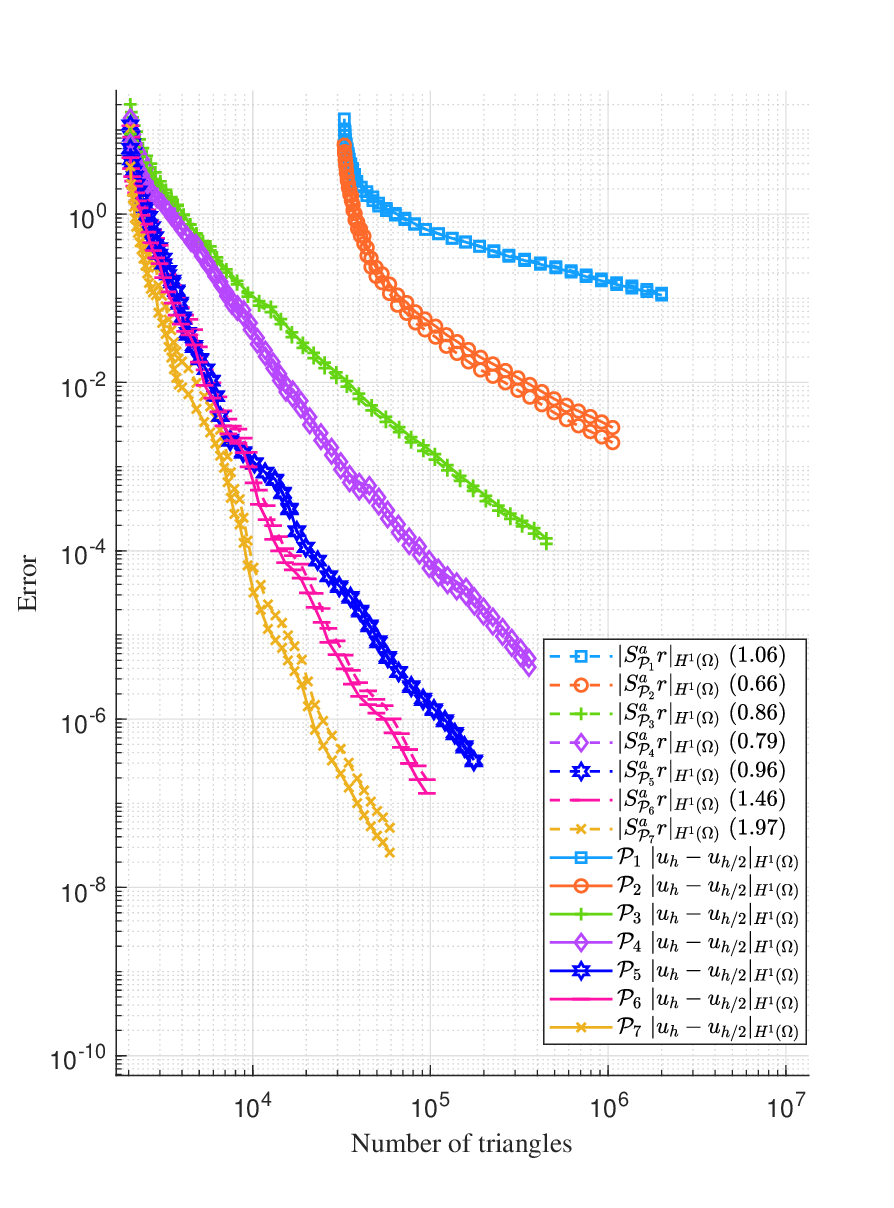} 
        \includegraphics[width=5.5cm]{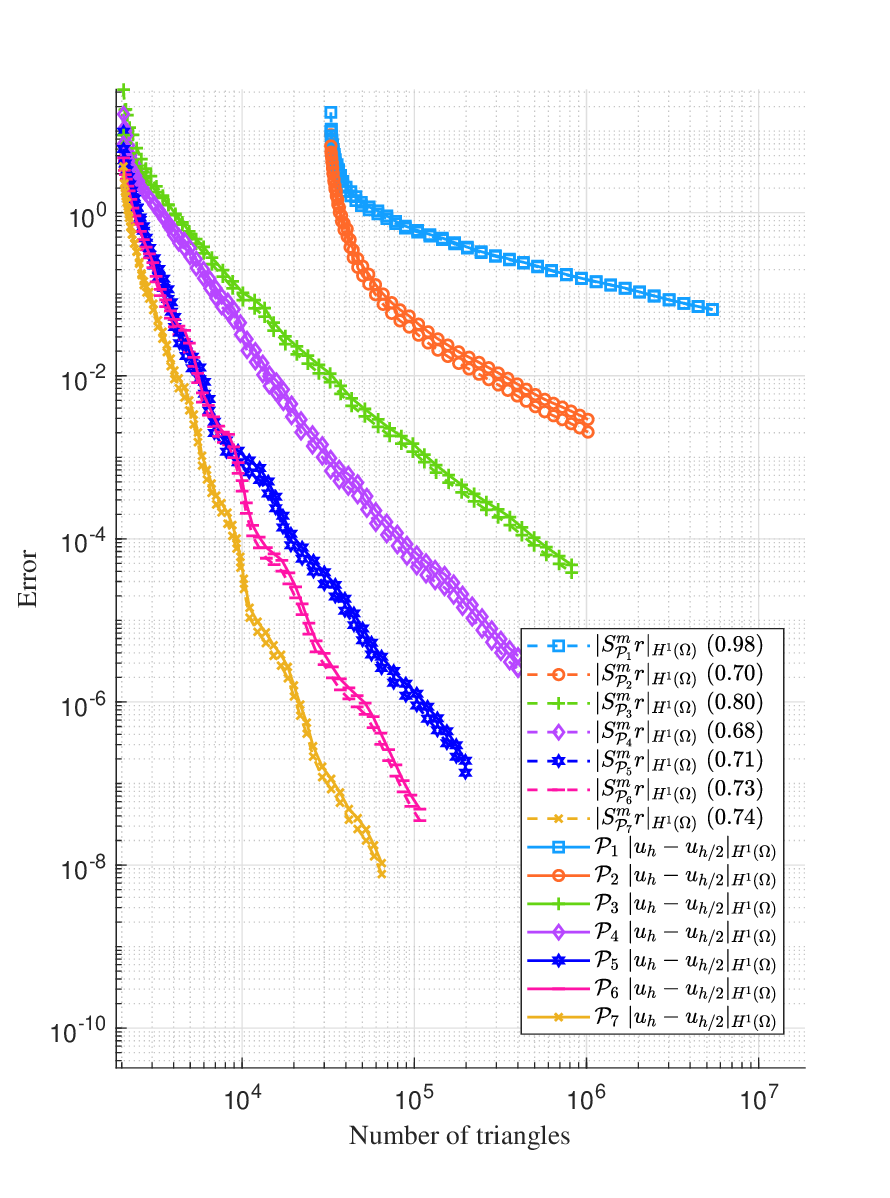}
    \caption{Convergence of $\mathcal{P}_p$-AFEM for the convection-diffusion problem driven by Jacobi (Left) and Gauss--Seidel (Right) smoother estimators with ratio estimator/error shown in the parenthesis.}
    \label{fig:Convection_diffusion_Pk}
\end{figure}

\begin{figure}[thp]
    \centering
    \includegraphics[width=10cm]{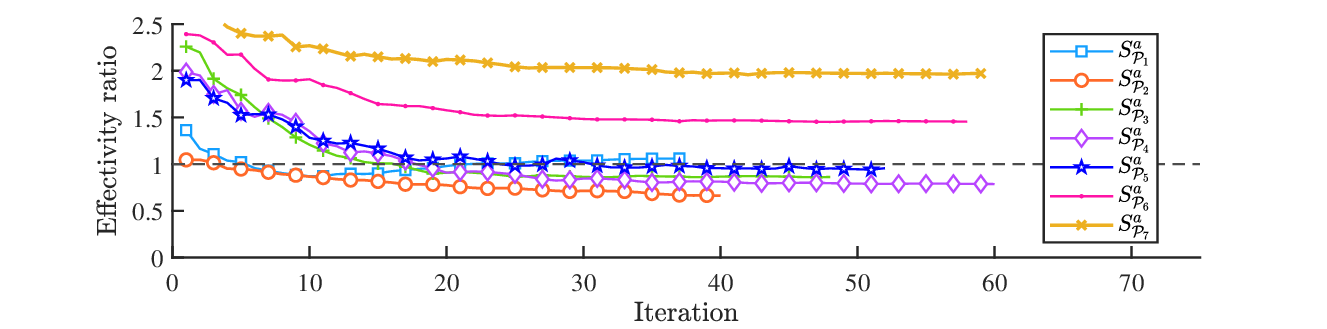} 
    \includegraphics[width=10cm]{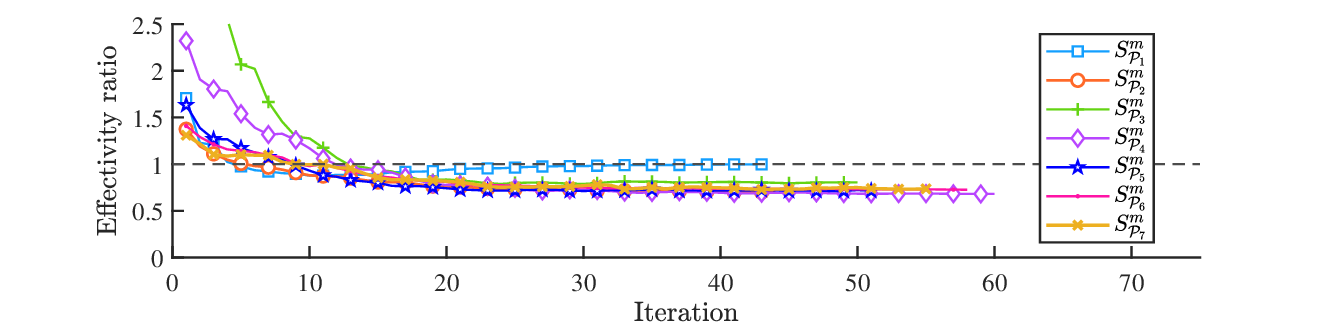}
    \caption{Effectivity ratios $|S^a_{\mathcal{P}_p} r|_{H^1(\Omega)} / | u_h-u_{h/2} |_{H^1(\Omega)}$ (top) and $|S^m_{\mathcal{P}_p} r|_{H^1(\Omega)} / | u_h-u_{h/2} |_{H^1(\Omega)}$ (bottom) for the convection-diffusion equation.}
    \label{fig:Convection_diffusion_ratio}
\end{figure}

\subsection{Helmholtz Equation} Another difficult non-symmetric problem is the numerical simulation of acoustic wave scattering in the high frequency regime. Let $\Omega_0\subset\Omega_1$ be a scattering body and $\Omega=\Omega_1\backslash\overline{\Omega}_0$ be the computational domain. The acoustic scattering is modeled by the Helmholtz equation
\begin{subequations}\label{eq:Helmholtz}
    \begin{align}
        -\Delta u-k^2u&=f\quad\text{ in }\Omega,\\
        \partial_nu-\texttt{i}ku&=g\quad\text{ on }\partial\Omega_1,\label{eq:Helmholtz2}\\
        u&=0\quad\text{ on }\partial\Omega_0.
    \end{align}
\end{subequations}
Here $k > 0$  denotes the wavenumber, $\texttt{i}=\sqrt{-1}$ is the imaginary unit,  and \eqref{eq:Helmholtz2} is an approximation to the Sommerfeld radiation condition at infinity. For large wavenumber $k\gg1$, the solution of \eqref{eq:Helmholtz} is highly oscillatory, leading to challenges in numerical approximation due to increased computational cost and pollution effects.  
{Let $H_{\partial\Omega_0}^1(\Omega, \mathbb{C}) = \{v \in H^1(\Omega, \mathbb{C}) : v|_{\partial\Omega_0} = 0\}$} be the complex-valued Sobolev space and $V_h\subset H_{{\partial\Omega_0}}^1(\Omega, \mathbb{C})$ be the continuous and piecewise $\mathcal{P}_p$ finite element subspace. The FEM for \eqref{eq:Helmholtz} seeks  $u_h \in V_h$  such that
\begin{equation*}\label{eq:Helmholtz_FEM}
\begin{aligned}
a(u_h, v_h) &=(\nabla u_h,\nabla v_h)-k^2(u_h,v_h)-\texttt{i}k(u_h,v_h)_{\partial\Omega_1}\\
&= (f,v_h)+(g,v_h)_{\partial\Omega_1}, \quad v_h \in V_h.   
\end{aligned}
\end{equation*}
Here $(\bullet,\bullet)$ and $(\bullet,\bullet)_{\partial\Omega_1}$ are complex-valued $L^2$ inner products.
A natural norm in the error analysis of \eqref{eq:Helmholtz_FEM} is $\|v\|_{1,k}=\sqrt{k^2\|v\|_{L^2(\Omega)}^2+|v|_{H^1(\Omega)}^2}$ (cf.~\cite{Wu2014,GongGrahamSpence2023}).
Similarly to Section \ref{subsect:convection_diffusion}, we use the Jacobi  $\|S^a_{\mathcal{P}_p}r\|_{1,k}$ and Gauss--Seidel $\|S^m_{\mathcal{P}_p}r\|_{1,k}$ a posteriori error estimates of AFEMs for \eqref{eq:Helmholtz}.

In the numerical experiment, we set $ \Omega_1 = (-0.5, 0.8) \times (-0.5, 0.5) $ and  $\Omega_0$ as a triangle with vertices $(0, 0)$,  $(0.5, \pm 0.5 \tan(\theta_0))$ with $\theta _0=\pi/6$. Let $(r,\theta)$ be the polar coordinates near the origin and $\alpha = \pi/(2\pi - 2\theta_0)$. The exact solution of \eqref{eq:Helmholtz} is $u = \phi(r) J_\alpha(kr) \sin\left( \alpha (\theta - \theta_0) \right) $, where $ J_\alpha$ is the Bessel function of the first kind and  
\begin{equation*}
    \phi(r) = \begin{cases}
    \left( 1 - \frac{r}{0.49} \right)^8, & r \leq 0.49, \\
    0, & r > 0.49
    \end{cases}
\end{equation*}
is a cutoff function. The polynomial degree $p$ ranges from 1 to 7 and the initial mesh size $h_0$ is $h_0 \approx 1/60$ for $p=1$ and $h_0 \approx 1/30$ for  $p\geq2$.
The experimental results in Figure \ref{fig:Helmholtz_error} demonstrate that the smoother-type estimators remain effective for the Helmholtz equation with large wave number $k=40\pi$. As shown in Figure \ref{fig:Helmholtz_mesh}, AFEMs based on smoother-type a posteriori error estimates correctly capture the oscillation and corner singularity of \eqref{eq:Helmholtz}. It is also observed from Figure \ref{fig:Helmholtz_error} that the effectivity ratios $\|S^a_{\mathcal{P}_p}r\|_{1,k}/\|u-u_h\|_{1,k}$ and $\|S^m_{\mathcal{P}_p}r\|_{1,k}/\|u-u_h\|_{1,k}$ are similar and uniform as the polynomial degree $p$ grows. Considering the computational cost,  $\|S^a_{\mathcal{P}_p}r\|_{1,k}$ is more efficient than its Gauss--Seidel counterpart.

\begin{figure}[thp]
\centering
\includegraphics[width=6cm]{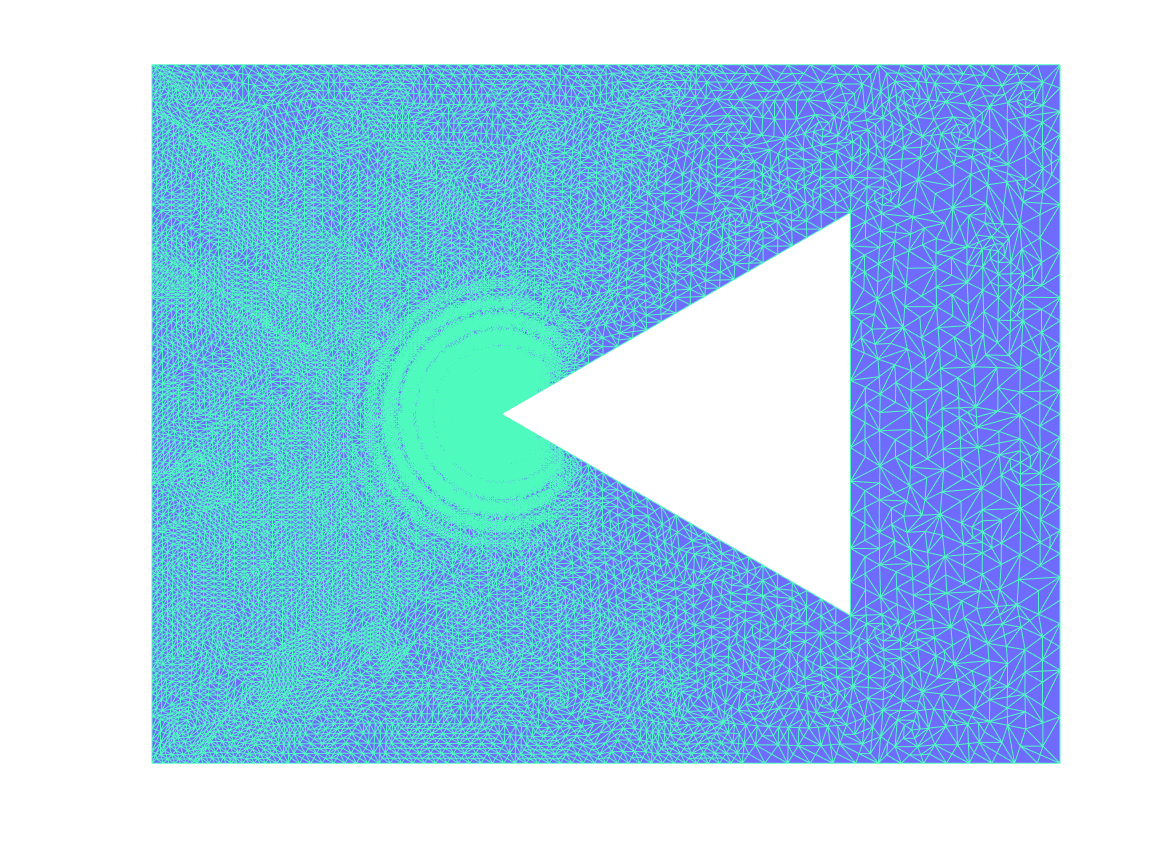} 
\includegraphics[width=6cm]{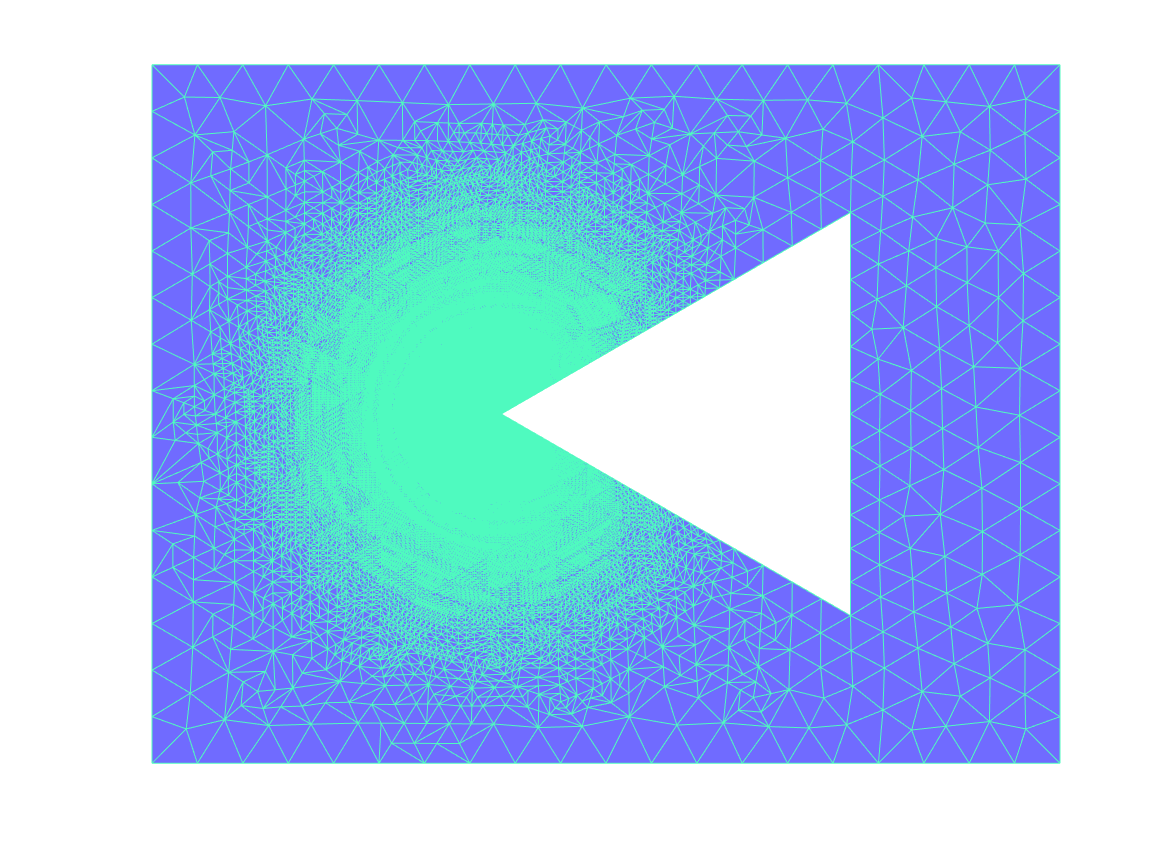}
    \caption{Adaptive meshes (approximately $10^5$ triangles) by $\mathcal{P}_1$-AFEM (left) and $\mathcal{P}_4$-AFEM (right) driven by the pointwise Jacobi smoother error estimator $\|S^a_{\mathcal{P}_p}r\|_{1,k}$.}
    \label{fig:Helmholtz_mesh}
\end{figure}

\begin{figure}[th]
    \centering
    \includegraphics[width=5.5cm]{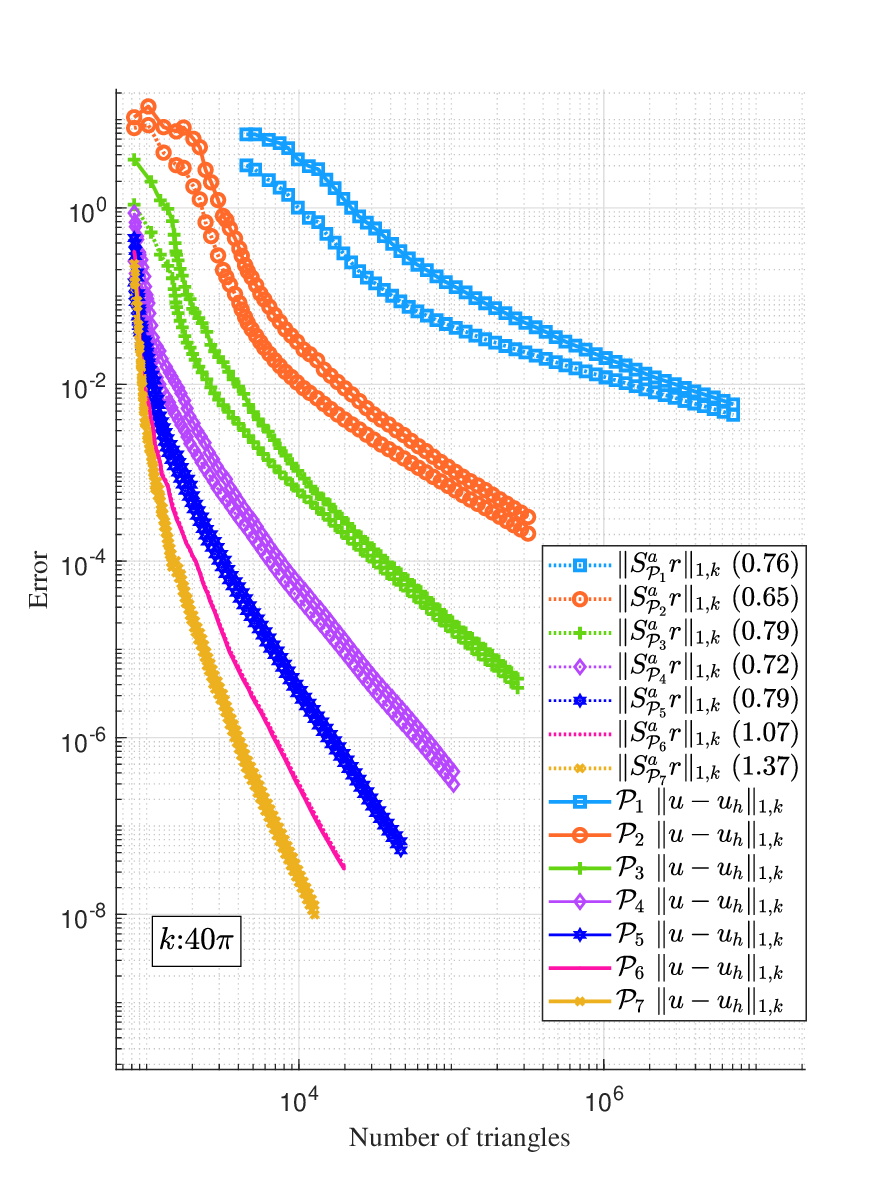} 
    \includegraphics[width=5.5cm]{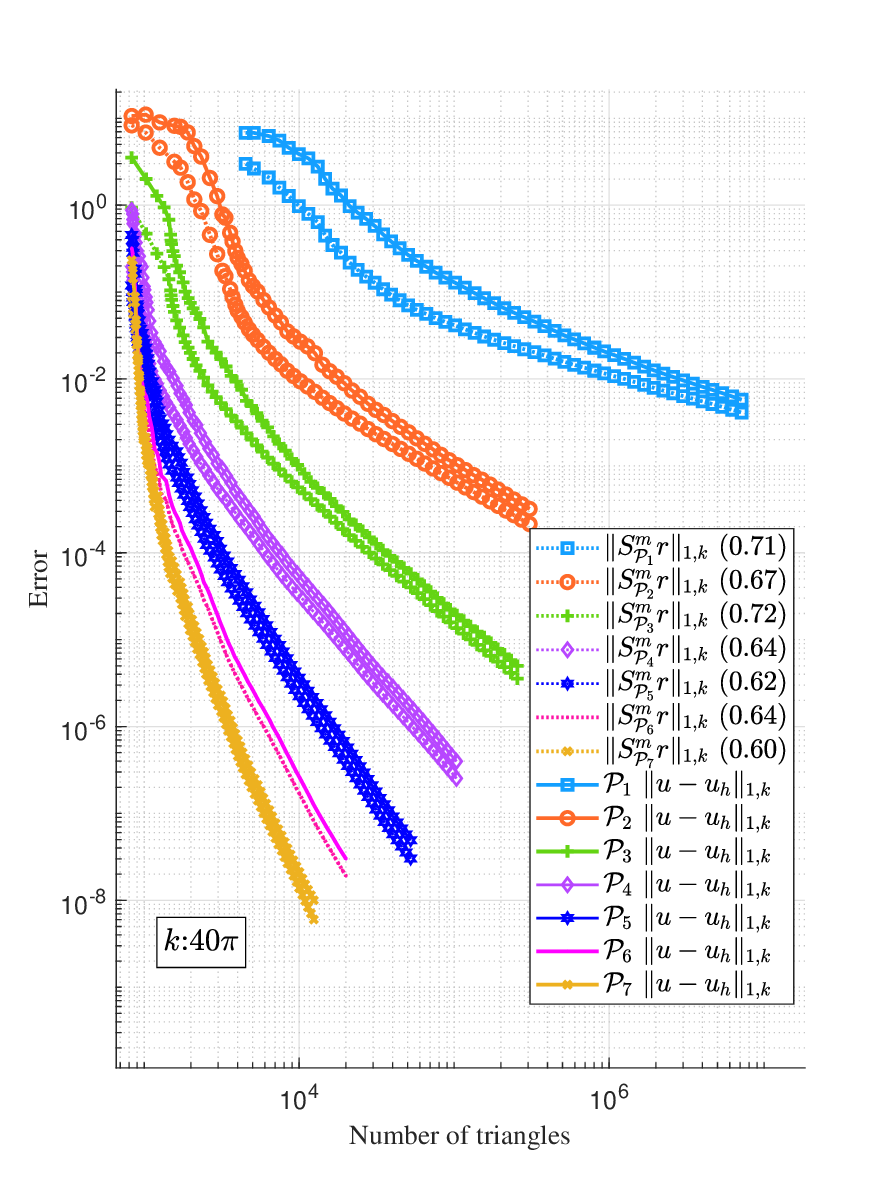}
    \caption{Convergence of $\mathcal{P}_p$-AFEM driven by Jacobi (Left) and Gauss--Seidel error estimators (Right) for the Helmholtz Equation with $k = 40\pi$ and estimator/error shown in the parenthesis.}
    \label{fig:Helmholtz_error}
\end{figure}

\begin{figure}[th]
    \centering
    \includegraphics[width=10cm]{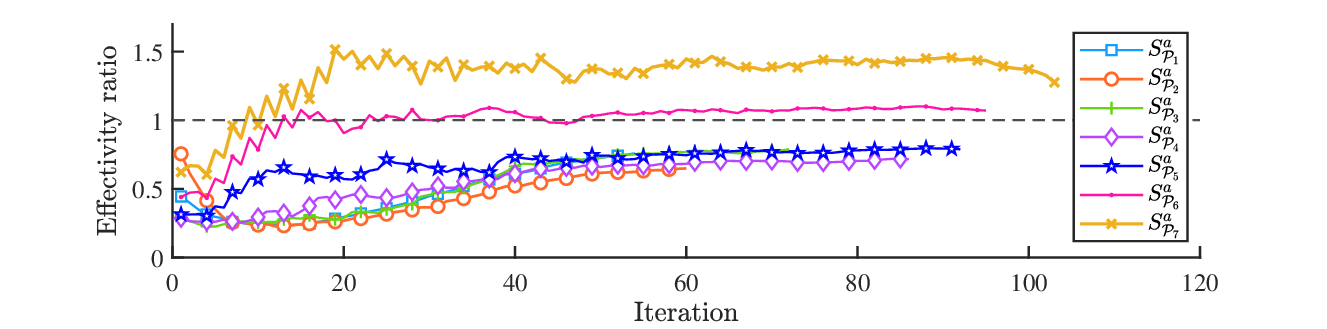} 
    \includegraphics[width=10cm]{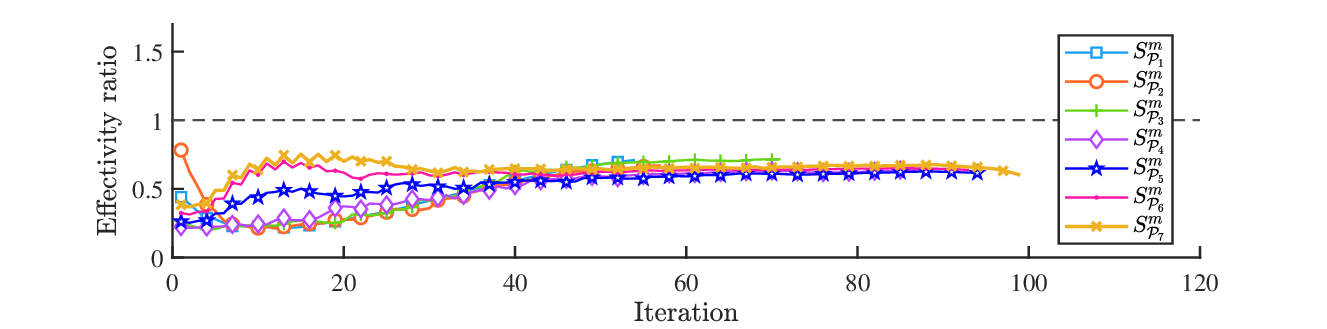}
    \caption{Effectivity ratios $\|S^a_{\mathcal{P}_p} r\|_{1,k} / \| u-u_h \|_{1,k}$ (top) and $\|S^m_{\mathcal{P}_p} r\|_{1,k} / \| u-u_h \|_{1,k}$ (bottom) for the Helmholtz equation.}
    \label{fig:Helmholtz_error_ratio}
\end{figure}

\section{Concluding Remarks}\label{sect:conclusion}
This paper develops a practical framework connecting a posteriori error estimates with smoothers of  iterative linear solvers. By leveraging smoothers of available preconditioners and linear solvers, we have derived a posteriori error estimates for various PDEs such as the Poisson, curl-curl, Biot, convection-diffusion and Helmholtz equations.  
For symmetric problems, reliability and efficiency of smoother-type error estimators have been theoretically verified with the help of two-level stable subspace decompositions and saturation assumptions. In addition, extensive numerical experiments demonstrate their robustness with respect to the polynomial degree.

The proposed analysis relies on the hierarchy $V_h\subset V_{h/2}\subset V$, which is available for standard conforming finite elements. However, $V_h\not\subset V_{h/2}\not \subset V$ for nonconforming finite elements and $V_h\not\subset V_{h/2}$ for some finite elements with extra smoothness, such as sophisticated structure-preserving elements for Stokes flow and elasticity.
In future work, we will explore  possible applications of smoother-type a posteriori error estimation to the aforementioned non-standard FEMs.

\section*{Acknowledgments} Y. L. would like to thank Prof. Shihua Gong for helpful discussion about iterative solvers for the Helmholtz equation. The authors would like to thank the anonymous referees for constructive remarks that  improved the quality and presentation of this paper. This work was supported by the National Key R\&D Program of China under grant 2024YFA1012600 and the National Natural Science Foundation of China under grant 12471346.

\bibliographystyle{elsarticle-num-names}

\begin{thebibliography}{74}
\expandafter\ifx\csname natexlab\endcsname\relax\def\natexlab#1{#1}\fi
\providecommand{\url}[1]{\texttt{#1}}
\providecommand{\href}[2]{#2}
\providecommand{\path}[1]{#1}
\providecommand{\DOIprefix}{doi:}
\providecommand{\ArXivprefix}{arXiv:}
\providecommand{\URLprefix}{URL: }
\providecommand{\Pubmedprefix}{pmid:}
\providecommand{\doi}[1]{\href{http://dx.doi.org/#1}{\path{#1}}}
\providecommand{\Pubmed}[1]{\href{pmid:#1}{\path{#1}}}
\providecommand{\bibinfo}[2]{#2}
\ifx\xfnm\relax \def\xfnm[#1]{\unskip,\space#1}\fi
\bibitem[{Ahmed et~al.(2019)Ahmed, Radu, and Nordbotten}]{AhmedRaduNordbotten2019}
\bibinfo{author}{E.~Ahmed}, \bibinfo{author}{F.~A. Radu}, \bibinfo{author}{J.~M. Nordbotten},
\newblock \bibinfo{title}{Adaptive poromechanics computations based on a posteriori error estimates for fully mixed formulations of {B}iot's consolidation model},
\newblock \bibinfo{journal}{Comput. Methods Appl. Mech. Engrg.} \bibinfo{volume}{347} (\bibinfo{year}{2019}) \bibinfo{pages}{264--294}. \DOIprefix\doi{10.1016/j.cma.2018.12.016}.
\bibitem[{Ainsworth(0708)}]{Ainsworth2007}
\bibinfo{author}{M.~Ainsworth},
\newblock \bibinfo{title}{A posteriori error estimation for lowest order {R}aviart-{T}homas mixed finite elements},
\newblock \bibinfo{journal}{SIAM J. Sci. Comput.} \bibinfo{volume}{30} (\bibinfo{year}{2007/08}) \bibinfo{pages}{189--204}. \DOIprefix\doi{10.1137/06067331X}.
\bibitem[{Ainsworth and Oden(2000)}]{AinsworthOden2000}
\bibinfo{author}{M.~Ainsworth}, \bibinfo{author}{J.~T. Oden}, \bibinfo{title}{A posteriori error estimation in finite element analysis}, Pure and Applied Mathematics (New York), \bibinfo{publisher}{Wiley-Interscience [John Wiley \& Sons], New York}, \bibinfo{year}{2000}.
\bibitem[{Arnold et~al.(2000)Arnold, Falk, and Winther}]{ArnoldFalkWinther2000}
\bibinfo{author}{D.~N. Arnold}, \bibinfo{author}{R.~S. Falk}, \bibinfo{author}{R.~Winther},
\newblock \bibinfo{title}{Multigrid in {H}(\text{div}) and {H}(\text{curl})},
\newblock \bibinfo{journal}{Numer. Math.} \bibinfo{volume}{85} (\bibinfo{year}{2000}) \bibinfo{pages}{197--217}.
\bibitem[{Babu\v{s}ka and Strouboulis(2001)}]{BabuskaStrouboulis2001}
\bibinfo{author}{I.~Babu\v{s}ka}, \bibinfo{author}{T.~Strouboulis}, \bibinfo{title}{The finite element method and its reliability}, Numerical Mathematics and Scientific Computation, \bibinfo{publisher}{The Clarendon Press, Oxford University Press, New York}, \bibinfo{year}{2001}.
\bibitem[{Bank(1996)}]{Bank1996}
\bibinfo{author}{R.~E. Bank},
\newblock \bibinfo{title}{Hierarchical bases and the finite element method},
\newblock in: \bibinfo{booktitle}{Acta numerica, 1996}, volume~\bibinfo{volume}{5} of \textit{\bibinfo{series}{Acta Numer.}}, \bibinfo{publisher}{Cambridge Univ. Press}, \bibinfo{address}{Cambridge}, \bibinfo{year}{1996}, pp. \bibinfo{pages}{1--43}. \DOIprefix\doi{10.1017/S0962492900002610}.
\bibitem[{Bank and Li(2019)}]{BankLi2019}
\bibinfo{author}{R.~E. Bank}, \bibinfo{author}{Y.~Li},
\newblock \bibinfo{title}{Superconvergent recovery of {R}aviart-{T}homas mixed finite elements on triangular grids},
\newblock \bibinfo{journal}{J. Sci. Comput.} \bibinfo{volume}{81} (\bibinfo{year}{2019}) \bibinfo{pages}{1882--1905}. \DOIprefix\doi{10.1007/s10915-019-01068-0}.
\bibitem[{Bank and Ovall(2017)}]{BankOvall2017}
\bibinfo{author}{R.~E. Bank}, \bibinfo{author}{J.~S. Ovall},
\newblock \bibinfo{title}{Some remarks on interpolation and best approximation},
\newblock \bibinfo{journal}{Numer. Math.} \bibinfo{volume}{137} (\bibinfo{year}{2017}) \bibinfo{pages}{289--302}. \DOIprefix\doi{10.1007/s00211-017-0877-7}.
\bibitem[{Bank and Smith(1993)}]{BankSmith1993}
\bibinfo{author}{R.~E. Bank}, \bibinfo{author}{R.~K. Smith},
\newblock \bibinfo{title}{A posteriori error estimates based on hierarchical bases},
\newblock \bibinfo{journal}{SIAM J. Numer. Anal.} \bibinfo{volume}{30} (\bibinfo{year}{1993}) \bibinfo{pages}{921--935}. \DOIprefix\doi{10.1137/0730048}.
\bibitem[{Bank and Weiser(1985)}]{BankWeiser1985}
\bibinfo{author}{R.~E. Bank}, \bibinfo{author}{A.~Weiser},
\newblock \bibinfo{title}{Some a posteriori error estimators for elliptic partial differential equations},
\newblock \bibinfo{journal}{Math. Comp.} \bibinfo{volume}{44} (\bibinfo{year}{1985}) \bibinfo{pages}{283--301}. \DOIprefix\doi{10.2307/2007953}.
\bibitem[{Bank and Xu(2003{\natexlab{a}})}]{BankXu2003a}
\bibinfo{author}{R.~E. Bank}, \bibinfo{author}{J.~Xu},
\newblock \bibinfo{title}{Asymptotically exact a posteriori error estimators. {I}. {G}rids with superconvergence},
\newblock \bibinfo{journal}{SIAM J. Numer. Anal.} \bibinfo{volume}{41} (\bibinfo{year}{2003}{\natexlab{a}}) \bibinfo{pages}{2294--2312}. \DOIprefix\doi{10.1137/S003614290139874X}.
\bibitem[{Bank and Xu(2003{\natexlab{b}})}]{BankXu2003b}
\bibinfo{author}{R.~E. Bank}, \bibinfo{author}{J.~Xu},
\newblock \bibinfo{title}{Asymptotically exact a posteriori error estimators. {II}. {G}eneral unstructured grids},
\newblock \bibinfo{journal}{SIAM J. Numer. Anal.} \bibinfo{volume}{41} (\bibinfo{year}{2003}{\natexlab{b}}) \bibinfo{pages}{2313--2332}. \DOIprefix\doi{10.1137/S0036142901398751}.
\bibitem[{Bank et~al.(2007)Bank, Xu, and Zheng}]{BankXuZheng2007}
\bibinfo{author}{R.~E. Bank}, \bibinfo{author}{J.~Xu}, \bibinfo{author}{B.~Zheng},
\newblock \bibinfo{title}{Superconvergent derivative recovery for {L}agrange triangular elements of degree $p$ on unstructured grids},
\newblock \bibinfo{journal}{SIAM J. Numer. Anal.} \bibinfo{volume}{45} (\bibinfo{year}{2007}) \bibinfo{pages}{2032--2046}.
\bibitem[{Bartels and Carstensen(2002)}]{BartelsCarstensen2002}
\bibinfo{author}{S.~Bartels}, \bibinfo{author}{C.~Carstensen},
\newblock \bibinfo{title}{Each averaging technique yields reliable a posteriori error control in {FEM} on unstructured grids. {II}. {H}igher order {FEM}},
\newblock \bibinfo{journal}{Math. Comp.} \bibinfo{volume}{71} (\bibinfo{year}{2002}) \bibinfo{pages}{971--994}. \DOIprefix\doi{10.1090/S0025-5718-02-01412-6}.
\bibitem[{Bastidas~Olivares et~al.(2025)Bastidas~Olivares, Beni~Hamad, Vohral\'ik, and Yotov}]{BastidasVohralik2025}
\bibinfo{author}{M.~Bastidas~Olivares}, \bibinfo{author}{A.~Beni~Hamad}, \bibinfo{author}{M.~Vohral\'ik}, \bibinfo{author}{I.~Yotov},
\newblock \bibinfo{title}{A posteriori algebraic error estimates and nonoverlapping domain decomposition in mixed formulations: energy coarse grid balancing, local mass conservation on each step, and line search},
\newblock \bibinfo{journal}{Comput. Methods Appl. Mech. Engrg.} \bibinfo{volume}{444} (\bibinfo{year}{2025}) \bibinfo{pages}{Paper No. 118090, 30}. \DOIprefix\doi{10.1016/j.cma.2025.118090}.
\bibitem[{Beck et~al.(2000)Beck, Hiptmair, Hoppe, and Wohlmuth}]{BHHW2000}
\bibinfo{author}{R.~Beck}, \bibinfo{author}{R.~Hiptmair}, \bibinfo{author}{R.~H.~W. Hoppe}, \bibinfo{author}{B.~Wohlmuth},
\newblock \bibinfo{title}{Residual based a posteriori error estimators for eddy current computation},
\newblock \bibinfo{journal}{M2AN Math. Model. Numer. Anal.} \bibinfo{volume}{34} (\bibinfo{year}{2000}) \bibinfo{pages}{159--182}. \DOIprefix\doi{10.1051/m2an:2000136}.
\bibitem[{Becker and Rannacher(2001)}]{BeckerRannacher2001}
\bibinfo{author}{R.~Becker}, \bibinfo{author}{R.~Rannacher},
\newblock \bibinfo{title}{An optimal control approach to a posteriori error estimation in finite element methods},
\newblock \bibinfo{journal}{Acta Numer.} \bibinfo{volume}{10} (\bibinfo{year}{2001}) \bibinfo{pages}{1--102}. \DOIprefix\doi{10.1017/S0962492901000010}.
\bibitem[{Bey(2000)}]{Bey2000}
\bibinfo{author}{J.~Bey},
\newblock \bibinfo{title}{Simplicial grid refinement: on {F}reudenthal's algorithm and the optimal number of congruence classes},
\newblock \bibinfo{journal}{Numer. Math.} \bibinfo{volume}{85} (\bibinfo{year}{2000}) \bibinfo{pages}{1--29}.
\bibitem[{Bonito et~al.(2024)Bonito, Canuto, Nochetto, and Veeser}]{BonitoCanutoNochettoVeeser2024}
\bibinfo{author}{A.~Bonito}, \bibinfo{author}{C.~Canuto}, \bibinfo{author}{R.~H. Nochetto}, \bibinfo{author}{A.~Veeser},
\newblock \bibinfo{title}{Adaptive finite element methods},
\newblock \bibinfo{journal}{Acta Numer.} \bibinfo{volume}{33} (\bibinfo{year}{2024}) \bibinfo{pages}{163--485}. \DOIprefix\doi{10.1017/S0962492924000011}.
\bibitem[{Boon et~al.(2021)Boon, Kuchta, Mardal, and Ruiz-Baier}]{BoonKuchtaMardalRuizBaier2021}
\bibinfo{author}{W.~M. Boon}, \bibinfo{author}{M.~Kuchta}, \bibinfo{author}{K.-A. Mardal}, \bibinfo{author}{R.~Ruiz-Baier},
\newblock \bibinfo{title}{Robust preconditioning of perturbed saddle-point problem and conservative discretizations of biot's equations utilizing total pressure},
\newblock \bibinfo{journal}{SIAM J. Sci. Comput.} \bibinfo{volume}{43} (\bibinfo{year}{2021}) \bibinfo{pages}{B961--B983}. \DOIprefix\doi{10.1137/20M1379708}.
\bibitem[{Braess et~al.(2009)Braess, Pillwein, and Sch\"oberl}]{BraessPillweinSchoberl2009}
\bibinfo{author}{D.~Braess}, \bibinfo{author}{V.~Pillwein}, \bibinfo{author}{J.~Sch\"oberl},
\newblock \bibinfo{title}{Equilibrated residual error estimates are {$p$}-robust},
\newblock \bibinfo{journal}{Comput. Methods Appl. Mech. Engrg.} \bibinfo{volume}{198} (\bibinfo{year}{2009}) \bibinfo{pages}{1189--1197}. \DOIprefix\doi{10.1016/j.cma.2008.12.010}.
\bibitem[{Braess and Verf\"urth(1996)}]{BV1996}
\bibinfo{author}{D.~Braess}, \bibinfo{author}{R.~Verf\"urth},
\newblock \bibinfo{title}{A posteriori error estimators for the {R}aviart--{T}homas element},
\newblock \bibinfo{journal}{SIAM J. Numer. Anal.} \bibinfo{volume}{33} (\bibinfo{year}{1996}) \bibinfo{pages}{2431--2444}.
\bibitem[{Canuto et~al.(2019)Canuto, Nochetto, Stevenson, and Verani}]{CanutoNochettoStevensonVerani2019}
\bibinfo{author}{C.~Canuto}, \bibinfo{author}{R.~H. Nochetto}, \bibinfo{author}{R.~P. Stevenson}, \bibinfo{author}{M.~Verani},
\newblock \bibinfo{title}{A saturation property for the spectral-{G}alerkin approximation of a {D}irichlet problem in a square},
\newblock \bibinfo{journal}{ESAIM Math. Model. Numer. Anal.} \bibinfo{volume}{53} (\bibinfo{year}{2019}) \bibinfo{pages}{987--1003}. \DOIprefix\doi{10.1051/m2an/2019015}.
\bibitem[{C\'{a}rcamo et~al.(2024)C\'{a}rcamo, Caiazzo, Galarce, and Mura}]{Carcamo2024}
\bibinfo{author}{C.~C\'{a}rcamo}, \bibinfo{author}{A.~Caiazzo}, \bibinfo{author}{F.~Galarce}, \bibinfo{author}{J.~Mura},
\newblock \bibinfo{title}{A stabilized total pressure-formulation of the biot’s poroelasticity equations in frequency domain: Numerical analysis and applications},
\newblock \bibinfo{journal}{Comput. Methods Appl. Mech. Engrg.} \bibinfo{volume}{432} (\bibinfo{year}{2024}) \bibinfo{pages}{117353}. \DOIprefix\doi{10.1016/j.cma.2024.117353}.
\bibitem[{Carstensen(1997)}]{Carstensen1997}
\bibinfo{author}{C.~Carstensen},
\newblock \bibinfo{title}{A posteriori error estimate for the mixed finite element method},
\newblock \bibinfo{journal}{Math. Comp.} \bibinfo{volume}{66} (\bibinfo{year}{1997}) \bibinfo{pages}{465--476}.
\bibitem[{Carstensen and Bartels(2002)}]{CarstensenBartels2002}
\bibinfo{author}{C.~Carstensen}, \bibinfo{author}{S.~Bartels},
\newblock \bibinfo{title}{Each averaging technique yields reliable a posteriori error control in {FEM} on unstructured grids. {I}. {L}ow order conforming, nonconforming, and mixed {FEM}},
\newblock \bibinfo{journal}{Math. Comp.} \bibinfo{volume}{71} (\bibinfo{year}{2002}) \bibinfo{pages}{945--969}. \DOIprefix\doi{10.1090/S0025-5718-02-01402-3}.
\bibitem[{Carstensen et~al.(2016)Carstensen, Gallistl, and Gedicke}]{CarstensenGallistlGedicke2016}
\bibinfo{author}{C.~Carstensen}, \bibinfo{author}{D.~Gallistl}, \bibinfo{author}{J.~Gedicke},
\newblock \bibinfo{title}{Justification of the saturation assumption},
\newblock \bibinfo{journal}{Numer. Math.} \bibinfo{volume}{134} (\bibinfo{year}{2016}) \bibinfo{pages}{1--25}. \DOIprefix\doi{10.1007/s00211-015-0769-7}.
\bibitem[{Cascon et~al.(2008)Cascon, Kreuzer, Nochetto, and Siebert}]{CKNS2008}
\bibinfo{author}{J.~M. Cascon}, \bibinfo{author}{C.~Kreuzer}, \bibinfo{author}{R.~H. Nochetto}, \bibinfo{author}{K.~G. Siebert},
\newblock \bibinfo{title}{Quasi-optimal convergence rate for an adaptive finite element method},
\newblock \bibinfo{journal}{SIAM J. Numer. Anal.} \bibinfo{volume}{46} (\bibinfo{year}{2008}) \bibinfo{pages}{2524--2550}.
\bibitem[{Chaumont-Frelet and Vohral\'ik(2023)}]{ChaumontVohralik2023}
\bibinfo{author}{T.~Chaumont-Frelet}, \bibinfo{author}{M.~Vohral\'ik},
\newblock \bibinfo{title}{{$p$}-robust equilibrated flux reconstruction in {$H({\rm curl})$} based on local minimizations: application to a posteriori analysis of the curl-curl problem},
\newblock \bibinfo{journal}{SIAM J. Numer. Anal.} \bibinfo{volume}{61} (\bibinfo{year}{2023}) \bibinfo{pages}{1783--1818}. \DOIprefix\doi{10.1137/21M141909X}.
\bibitem[{Demkowicz(2007)}]{Demkowicz2007}
\bibinfo{author}{L.~Demkowicz}, \bibinfo{title}{Computing with {$hp$}-adaptive finite elements. {V}ol. 1}, Chapman \& Hall/CRC Applied Mathematics and Nonlinear Science Series, \bibinfo{publisher}{Chapman \& Hall/CRC, Boca Raton, FL}, \bibinfo{year}{2007}. \DOIprefix\doi{10.1201/9781420011692}, \bibinfo{note}{one and two dimensional elliptic and Maxwell problems, With 1 CD-ROM (UNIX)}.
\bibitem[{Demkowicz et~al.(2008)Demkowicz, Kurtz, Pardo, Paszy\'nski, Rachowicz, and Zdunek}]{Demkowicz2008}
\bibinfo{author}{L.~Demkowicz}, \bibinfo{author}{J.~Kurtz}, \bibinfo{author}{D.~Pardo}, \bibinfo{author}{M.~Paszy\'nski}, \bibinfo{author}{W.~Rachowicz}, \bibinfo{author}{A.~Zdunek}, \bibinfo{title}{Computing with {$hp$}-adaptive finite elements. {V}ol. 2}, Chapman \& Hall/CRC Applied Mathematics and Nonlinear Science Series, \bibinfo{publisher}{Chapman \& Hall/CRC, Boca Raton, FL}, \bibinfo{year}{2008}. \bibinfo{note}{Frontiers: three dimensional elliptic and Maxwell problems with applications}.
\bibitem[{Demlow and Stevenson(2011)}]{DemlowStevenson2011}
\bibinfo{author}{A.~Demlow}, \bibinfo{author}{R.~Stevenson},
\newblock \bibinfo{title}{Convergence and quasi-optimality of an adaptive finite element method for controlling ${L}^2$ errors},
\newblock \bibinfo{journal}{Numer. Math.} \bibinfo{volume}{117} (\bibinfo{year}{2011}) \bibinfo{pages}{185--218}.
\bibitem[{D\"{o}rfler(1996)}]{Dorfler1996}
\bibinfo{author}{W.~D\"{o}rfler},
\newblock \bibinfo{title}{A convergent adaptive algorithm for {P}oisson's equation},
\newblock \bibinfo{journal}{SIAM J. Numer. Anal.} \bibinfo{volume}{33} (\bibinfo{year}{1996}) \bibinfo{pages}{1106--1124}. \DOIprefix\doi{10.1137/0733054}.
\bibitem[{D\"orfler and Nochetto(2002)}]{DorflerNochetto2002}
\bibinfo{author}{W.~D\"orfler}, \bibinfo{author}{R.~H. Nochetto},
\newblock \bibinfo{title}{Small data oscillation implies the saturation assumption},
\newblock \bibinfo{journal}{Numer. Math.} \bibinfo{volume}{91} (\bibinfo{year}{2002}) \bibinfo{pages}{1--12}. \DOIprefix\doi{10.1007/s002110100321}.
\bibitem[{Ern and Meunier(2009)}]{ErnMeunier2009}
\bibinfo{author}{A.~Ern}, \bibinfo{author}{S.~Meunier},
\newblock \bibinfo{title}{A~posteriori error analysis of {E}uler-{G}alerkin approximations to coupled elliptic-parabolic problems},
\newblock \bibinfo{journal}{M2AN Math. Model. Numer. Anal.} \bibinfo{volume}{43} (\bibinfo{year}{2009}) \bibinfo{pages}{353--375}. \DOIprefix\doi{10.1051/m2an:2008048}.
\bibitem[{Ern and Vohral\'ik(2015)}]{ErnVohralik2015}
\bibinfo{author}{A.~Ern}, \bibinfo{author}{M.~Vohral\'ik},
\newblock \bibinfo{title}{Polynomial-degree-robust a posteriori estimates in a unified setting for conforming, nonconforming, discontinuous {G}alerkin, and mixed discretizations},
\newblock \bibinfo{journal}{SIAM J. Numer. Anal.} \bibinfo{volume}{53} (\bibinfo{year}{2015}) \bibinfo{pages}{1058--1081}. \DOIprefix\doi{10.1137/1309501}.
\bibitem[{Ern and Vohral\'ik(2020)}]{ErnVohralik2020}
\bibinfo{author}{A.~Ern}, \bibinfo{author}{M.~Vohral\'ik},
\newblock \bibinfo{title}{Stable broken {$H^1$} and {$H({\rm div})$} polynomial extensions for polynomial-degree-robust potential and flux reconstruction in three space dimensions},
\newblock \bibinfo{journal}{Math. Comp.} \bibinfo{volume}{89} (\bibinfo{year}{2020}) \bibinfo{pages}{551--594}. \DOIprefix\doi{10.1090/mcom/3482}.
\bibitem[{Feischl et~al.(2014)Feischl, Page, and Praetorius}]{FeischlPagePraetorius2014}
\bibinfo{author}{M.~Feischl}, \bibinfo{author}{M.~Page}, \bibinfo{author}{D.~Praetorius},
\newblock \bibinfo{title}{Convergence and quasi-optimality of adaptive {FEM} with inhomogeneous {D}irichlet data},
\newblock \bibinfo{journal}{J. Comput. Appl. Math.} \bibinfo{volume}{255} (\bibinfo{year}{2014}) \bibinfo{pages}{481--501}. \DOIprefix\doi{10.1016/j.cam.2013.06.009}.
\bibitem[{Ferraz-Leite et~al.(2010)Ferraz-Leite, Ortner, and Praetorius}]{FerrazOrtnerPraetorious2010}
\bibinfo{author}{S.~Ferraz-Leite}, \bibinfo{author}{C.~Ortner}, \bibinfo{author}{D.~Praetorius},
\newblock \bibinfo{title}{Convergence of simple adaptive {G}alerkin schemes based on {$h-h/2$} error estimators},
\newblock \bibinfo{journal}{Numer. Math.} \bibinfo{volume}{116} (\bibinfo{year}{2010}) \bibinfo{pages}{291--316}. \DOIprefix\doi{10.1007/s00211-010-0292-9}.
\bibitem[{Ferraz-Leite and Praetorius(2008)}]{FerrazPraetorius2008}
\bibinfo{author}{S.~Ferraz-Leite}, \bibinfo{author}{D.~Praetorius},
\newblock \bibinfo{title}{Simple a posteriori error estimators for the {$h$}-version of the boundary element method},
\newblock \bibinfo{journal}{Computing} \bibinfo{volume}{83} (\bibinfo{year}{2008}) \bibinfo{pages}{135--162}. \DOIprefix\doi{10.1007/s00607-008-0017-4}.
\bibitem[{Fumagalli et~al.(2025)Fumagalli, Parolini, and Verani}]{FumagalliParoliniVerani2025}
\bibinfo{author}{I.~Fumagalli}, \bibinfo{author}{N.~Parolini}, \bibinfo{author}{M.~Verani},
\newblock \bibinfo{title}{Robust a posteriori error estimation for mixed finite element approximation of linear poroelasticity},
\newblock \bibinfo{journal}{J. Sci. Comput.} \bibinfo{volume}{103} (\bibinfo{year}{2025}) \bibinfo{pages}{103:22}. \DOIprefix\doi{10.1007/s10915-025-02814-3}.
\bibitem[{Gong et~al.(2023)Gong, Graham, and Spence}]{GongGrahamSpence2023}
\bibinfo{author}{S.~Gong}, \bibinfo{author}{I.~G. Graham}, \bibinfo{author}{E.~A. Spence},
\newblock \bibinfo{title}{Convergence of restricted additive {S}chwarz with impedance transmission conditions for discretised {H}elmholtz problems},
\newblock \bibinfo{journal}{Math. Comp.} \bibinfo{volume}{92} (\bibinfo{year}{2023}) \bibinfo{pages}{175--215}. \DOIprefix\doi{10.1090/mcom/3772}.
\bibitem[{Hiptmair(1999)}]{Hiptmair1999SINUM}
\bibinfo{author}{R.~Hiptmair},
\newblock \bibinfo{title}{Multigrid method for {M}axwell's equations},
\newblock \bibinfo{journal}{SIAM J. Numer. Anal.} \bibinfo{volume}{36} (\bibinfo{year}{1999}) \bibinfo{pages}{204--225}. \DOIprefix\doi{10.1137/S0036142997326203}.
\bibitem[{Hong and Kraus(2018)}]{HongKraus2018}
\bibinfo{author}{Q.~Hong}, \bibinfo{author}{J.~Kraus},
\newblock \bibinfo{title}{Parameter-robust stability of classical three-field formulation of biot’s consolidation model},
\newblock \bibinfo{journal}{Electron. Trans. Numer. Anal.} \bibinfo{volume}{48} (\bibinfo{year}{2018}) \bibinfo{pages}{202--226}.
\bibitem[{Hu et~al.(2021)Hu, Wu, and Zikatanov}]{HuWuZikatanov2021}
\bibinfo{author}{X.~Hu}, \bibinfo{author}{K.~Wu}, \bibinfo{author}{L.~T. Zikatanov},
\newblock \bibinfo{title}{A posteriori error estimates for multilevel methods for graph {L}aplacians},
\newblock \bibinfo{journal}{SIAM J. Sci. Comput.} \bibinfo{volume}{43} (\bibinfo{year}{2021}) \bibinfo{pages}{S727--S742}. \DOIprefix\doi{10.1137/20M1349618}.
\bibitem[{Jung and R{\"u}de(1996)}]{jung1996implicit}
\bibinfo{author}{M.~Jung}, \bibinfo{author}{U.~R{\"u}de},
\newblock \bibinfo{title}{Implicit extrapolation methods for multilevel finite element computations},
\newblock \bibinfo{journal}{SIAM Journal on Scientific Computing} \bibinfo{volume}{17} (\bibinfo{year}{1996}) \bibinfo{pages}{156--179}.
\bibitem[{Khan and Silvester(2021)}]{KhanSilvester2021}
\bibinfo{author}{A.~Khan}, \bibinfo{author}{D.~J. Silvester},
\newblock \bibinfo{title}{Robust a posteriori error estimation for mixed finite element approximation of linear poroelasticity},
\newblock \bibinfo{journal}{IMA J. Numer. Anal.} \bibinfo{volume}{41} (\bibinfo{year}{2021}) \bibinfo{pages}{2000--2025}. \DOIprefix\doi{10.1093/imanum/draa058}.
\bibitem[{Lee et~al.(2017)Lee, Mardal, and Winther}]{LeeMardalWinther2017}
\bibinfo{author}{J.~J. Lee}, \bibinfo{author}{K.-A. Mardal}, \bibinfo{author}{R.~Winther},
\newblock \bibinfo{title}{Parameter-robust discretization and preconditioning of biot's consolidation model},
\newblock \bibinfo{journal}{SIAM J. Sci. Comput.} \bibinfo{volume}{39} (\bibinfo{year}{2017}) \bibinfo{pages}{A1--A24}. \DOIprefix\doi{10.1137/15M1029473}.
\bibitem[{Li(2021{\natexlab{a}})}]{Li2021M2AN}
\bibinfo{author}{Y.~Li},
\newblock \bibinfo{title}{Quasi-optimal adaptive hybridized mixed finite element methods for linear elasticity},
\newblock \bibinfo{journal}{ESAIM Math. Model. Numer. Anal.} \bibinfo{volume}{55} (\bibinfo{year}{2021}{\natexlab{a}}) \bibinfo{pages}{1921--1938}. \DOIprefix\doi{10.1051/m2an/2021048}.
\bibitem[{Li(2021{\natexlab{b}})}]{Li2021MCOM}
\bibinfo{author}{Y.~Li},
\newblock \bibinfo{title}{Quasi-optimal adaptive mixed finite element methods for controlling natural norm errors},
\newblock \bibinfo{journal}{Math. Comp.} \bibinfo{volume}{90} (\bibinfo{year}{2021}{\natexlab{b}}) \bibinfo{pages}{565--593}. \DOIprefix\doi{10.1090/mcom/3590}.
\bibitem[{Li(2021{\natexlab{c}})}]{Li2021JSCb}
\bibinfo{author}{Y.~Li},
\newblock \bibinfo{title}{Recovery-based a posteriori error analysis for plate bending problems},
\newblock \bibinfo{journal}{J. Sci. Comput.} \bibinfo{volume}{88} (\bibinfo{year}{2021}{\natexlab{c}}) \bibinfo{pages}{Paper No. 77, 26}. \DOIprefix\doi{10.1007/s10915-021-01595-9}.
\bibitem[{Li(2025)}]{Li2025arxiv}
\bibinfo{author}{Y.~Li},
\newblock \bibinfo{title}{Some p-robust a posteriori error estimates based on auxiliary spaces},
\newblock \bibinfo{journal}{arXiv:2511.06603}  (\bibinfo{year}{2025}).
\bibitem[{Li and Zikatanov(2021)}]{LiZikatanov2021CAMWA}
\bibinfo{author}{Y.~Li}, \bibinfo{author}{L.~Zikatanov},
\newblock \bibinfo{title}{A posteriori error estimates of finite element methods by preconditioning},
\newblock \bibinfo{journal}{Comput. Math. Appl.} \bibinfo{volume}{91} (\bibinfo{year}{2021}) \bibinfo{pages}{192--201}. \DOIprefix\doi{10.1016/j.camwa.2020.08.001}.
\bibitem[{Li and Zikatanov(2025)}]{LiZikatanov2025mcom}
\bibinfo{author}{Y.~Li}, \bibinfo{author}{L.~Zikatanov},
\newblock \bibinfo{title}{Nodal auxiliary a posteriori error estimates},
\newblock \bibinfo{journal}{Math. Comp.}  (\bibinfo{year}{2025}) \bibinfo{pages}{published online}. \DOIprefix\doi{10.1090/mcom/4141}.
\bibitem[{{Li} and {Zikatanov}(2022)}]{LiZikatanov2022IMA}
\bibinfo{author}{Y.~{Li}}, \bibinfo{author}{L.~T. {Zikatanov}},
\newblock \bibinfo{title}{Residual-based a posteriori error estimates of mixed methods for a three-field {B}iot’s consolidation model},
\newblock \bibinfo{journal}{IMA J. Numer. Anal.} \bibinfo{volume}{42} (\bibinfo{year}{2022}) \bibinfo{pages}{602--648}. \DOIprefix\doi{10.1093/imanum/draa074}.
\bibitem[{Li(2018)}]{Li2018SINUM}
\bibinfo{author}{Y.-W. Li},
\newblock \bibinfo{title}{Global superconvergence of the lowest-order mixed finite element on mildly structured meshes},
\newblock \bibinfo{journal}{SIAM J. Numer. Anal.} \bibinfo{volume}{56} (\bibinfo{year}{2018}) \bibinfo{pages}{792--815}. \DOIprefix\doi{10.1137/17M112587X}.
\bibitem[{Mardal and Winther(2011)}]{MardalWinther2011}
\bibinfo{author}{K.-A. Mardal}, \bibinfo{author}{R.~Winther},
\newblock \bibinfo{title}{Preconditioning discretizations of systems of partial differential equations},
\newblock \bibinfo{journal}{Numer. Linear Algebra Appl.} \bibinfo{volume}{18} (\bibinfo{year}{2011}) \bibinfo{pages}{1--40}. \DOIprefix\doi{10.1002/nla.716}.
\bibitem[{Morin et~al.(2003)Morin, Nochetto, and Siebert}]{MNS2003}
\bibinfo{author}{P.~Morin}, \bibinfo{author}{R.~H. Nochetto}, \bibinfo{author}{K.~G. Siebert},
\newblock \bibinfo{title}{Local problems on stars: a posteriori error estimators, convergence, and performance},
\newblock \bibinfo{journal}{Math. Comp.} \bibinfo{volume}{72} (\bibinfo{year}{2003}) \bibinfo{pages}{1067--1097}. \DOIprefix\doi{10.1090/S0025-5718-02-01463-1}.
\bibitem[{Mulita et~al.(2021)Mulita, Giani, and Heltai}]{MulitaGianiHeltai2021}
\bibinfo{author}{O.~Mulita}, \bibinfo{author}{S.~Giani}, \bibinfo{author}{L.~Heltai},
\newblock \bibinfo{title}{Quasi-optimal mesh sequence construction through smoothed adaptive finite element methods},
\newblock \bibinfo{journal}{SIAM J. Sci. Comput.} \bibinfo{volume}{43} (\bibinfo{year}{2021}) \bibinfo{pages}{A2211--A2241}. \DOIprefix\doi{10.1137/19M1262097}.
\bibitem[{P\'e{} de~la Riva et~al.(2025)P\'e{} de~la Riva, Gaspar, Hu, Adler, Rodrigo, and Zikatanov}]{PGHARZ2025}
\bibinfo{author}{A.~P\'e{} de~la Riva}, \bibinfo{author}{F.~J. Gaspar}, \bibinfo{author}{X.~Hu}, \bibinfo{author}{J.~H. Adler}, \bibinfo{author}{C.~Rodrigo}, \bibinfo{author}{L.~T. Zikatanov},
\newblock \bibinfo{title}{Oscillation-free numerical schemes for {B}iot's model and their iterative coupling solution},
\newblock \bibinfo{journal}{SIAM J. Sci. Comput.} \bibinfo{volume}{47} (\bibinfo{year}{2025}) \bibinfo{pages}{A1809--A1834}. \DOIprefix\doi{10.1137/24M1680775}.
\bibitem[{Phillips and Wheeler(2007)}]{PhillipsWheeler2007}
\bibinfo{author}{P.~J. Phillips}, \bibinfo{author}{M.~F. Wheeler},
\newblock \bibinfo{title}{A coupling of mixed and continuous {G}alerkin finite element methods for poroelasticity. {I}. {T}he continuous in time case},
\newblock \bibinfo{journal}{Comput. Geosci.} \bibinfo{volume}{11} (\bibinfo{year}{2007}) \bibinfo{pages}{131--144}. \DOIprefix\doi{10.1007/s10596-007-9045-y}.
\bibitem[{Repin(2008)}]{Repin2008}
\bibinfo{author}{S.~Repin}, \bibinfo{title}{A posteriori estimates for partial differential equations}, volume~\bibinfo{volume}{4} of \textit{\bibinfo{series}{Radon Series on Computational and Applied Mathematics}}, \bibinfo{publisher}{Walter de Gruyter GmbH \& Co. KG, Berlin}, \bibinfo{year}{2008}. \DOIprefix\doi{10.1515/9783110203042}.
\bibitem[{Rodrigo et~al.(2018)Rodrigo, Hu, Ohm, Adler, Gaspar, and Zikatanov}]{RodrigoHu2018}
\bibinfo{author}{C.~Rodrigo}, \bibinfo{author}{X.~Hu}, \bibinfo{author}{P.~Ohm}, \bibinfo{author}{J.~H. Adler}, \bibinfo{author}{F.~J. Gaspar}, \bibinfo{author}{L.~T. Zikatanov},
\newblock \bibinfo{title}{New stabilized discretizations for poroelasticity and the {S}tokes' equations},
\newblock \bibinfo{journal}{Comput. Methods Appl. Mech. Engrg.} \bibinfo{volume}{341} (\bibinfo{year}{2018}) \bibinfo{pages}{467--484}. \DOIprefix\doi{10.1016/j.cma.2018.07.003}.
\bibitem[{Sch\"{o}berl(2008)}]{Schoberl2008}
\bibinfo{author}{J.~Sch\"{o}berl},
\newblock \bibinfo{title}{A posteriori error estimates for {M}axwell equations},
\newblock \bibinfo{journal}{Math. Comp.} \bibinfo{volume}{77} (\bibinfo{year}{2008}) \bibinfo{pages}{633--649}. \DOIprefix\doi{10.1090/S0025-5718-07-02030-3}.
\bibitem[{Schwab(1998)}]{Schwab1998}
\bibinfo{author}{C.~Schwab}, \bibinfo{title}{{$p$}- and {$hp$}-finite element methods}, Numerical Mathematics and Scientific Computation, \bibinfo{publisher}{The Clarendon Press, Oxford University Press, New York}, \bibinfo{year}{1998}. \bibinfo{note}{Theory and applications in solid and fluid mechanics}.
\bibitem[{Verf\"{u}rth(2013)}]{Verfurth2013}
\bibinfo{author}{R.~Verf\"{u}rth}, \bibinfo{title}{A posteriori error estimation techniques for finite element methods}, Numerical Mathematics and Scientific Computation, \bibinfo{publisher}{Oxford University Press}, \bibinfo{address}{Oxford}, \bibinfo{year}{2013}. \DOIprefix\doi{10.1093/acprof:oso/9780199679423.001.0001}.
\bibitem[{Wheeler et~al.(2022)Wheeler, Girault, and Li}]{WheelerGiraultLi2022}
\bibinfo{author}{M.~F. Wheeler}, \bibinfo{author}{V.~Girault}, \bibinfo{author}{H.~Li},
\newblock \bibinfo{title}{A posteriori error estimates for {B}iot system using a mixed discretization for flow},
\newblock \bibinfo{journal}{Comput. Methods Appl. Mech. Engrg.} \bibinfo{volume}{402} (\bibinfo{year}{2022}) \bibinfo{pages}{115240}. \DOIprefix\doi{10.1016/j.cma.2022.115240}.
\bibitem[{Wu(2014)}]{Wu2014}
\bibinfo{author}{H.~Wu},
\newblock \bibinfo{title}{Pre-asymptotic error analysis of {CIP}-{FEM} and {FEM} for the {H}elmholtz equation with high wave number. {P}art {I}: linear version},
\newblock \bibinfo{journal}{IMA J. Numer. Anal.} \bibinfo{volume}{34} (\bibinfo{year}{2014}) \bibinfo{pages}{1266--1288}. \DOIprefix\doi{10.1093/imanum/drt033}.
\bibitem[{Xu(1992)}]{Xu1992}
\bibinfo{author}{J.~Xu},
\newblock \bibinfo{title}{Iterative methods by space decomposition and subspace correction},
\newblock \bibinfo{journal}{SIAM Rev.} \bibinfo{volume}{34} (\bibinfo{year}{1992}) \bibinfo{pages}{581--613}. \DOIprefix\doi{10.1137/1034116}.
\bibitem[{Xu and Zikatanov(2002)}]{XuZikatanov2002}
\bibinfo{author}{J.~Xu}, \bibinfo{author}{L.~Zikatanov},
\newblock \bibinfo{title}{The method of alternating projections and the method of subspace corrections in {H}ilbert space},
\newblock \bibinfo{journal}{J. Amer. Math. Soc.} \bibinfo{volume}{15} (\bibinfo{year}{2002}) \bibinfo{pages}{573--597}. \DOIprefix\doi{10.1090/S0894-0347-02-00398-3}.
\bibitem[{Xu and Zikatanov(2017)}]{XuZikatanov2017}
\bibinfo{author}{J.~Xu}, \bibinfo{author}{L.~Zikatanov},
\newblock \bibinfo{title}{Algebraic multigrid methods},
\newblock \bibinfo{journal}{Acta Numer.} \bibinfo{volume}{26} (\bibinfo{year}{2017}) \bibinfo{pages}{591--721}. \DOIprefix\doi{10.1017/S0962492917000083}.
\bibitem[{Xu and Zikatanov(2018)}]{XuZikatanov2018}
\bibinfo{author}{W.~Xu}, \bibinfo{author}{L.~T. Zikatanov},
\newblock \bibinfo{title}{Adaptive aggregation on graphs},
\newblock \bibinfo{journal}{J. Comput. Appl. Math.} \bibinfo{volume}{340} (\bibinfo{year}{2018}) \bibinfo{pages}{718--730}. \DOIprefix\doi{10.1016/j.cam.2017.10.032}.
\bibitem[{Zhang and Naga(2005)}]{ZhangNaga2005}
\bibinfo{author}{Z.~Zhang}, \bibinfo{author}{A.~Naga},
\newblock \bibinfo{title}{A new finite element gradient recovery method: superconvergence property},
\newblock \bibinfo{journal}{SIAM J. Sci. Comput.} \bibinfo{volume}{26} (\bibinfo{year}{2005}) \bibinfo{pages}{1192--1213}. \DOIprefix\doi{10.1137/S1064827503402837}.
\bibitem[{Zienkiewicz and Zhu(1992)}]{ZienkiewiczZhu1992b}
\bibinfo{author}{O.~C. Zienkiewicz}, \bibinfo{author}{J.~Z. Zhu},
\newblock \bibinfo{title}{The superconvergent patch recovery and a posteriori error estimates. {II}. error estimates and adaptivity},
\newblock \bibinfo{journal}{Internat. J. Numer. Methods Engrg.} \bibinfo{volume}{33} (\bibinfo{year}{1992}) \bibinfo{pages}{1365--1382}.

\end{thebibliography}

\end{document}